\documentclass[11pt]{amsart}

\usepackage[a4paper,hmargin=3.5cm,vmargin=3.5cm]{geometry}
\usepackage{amsfonts,amssymb,amscd,amstext}
\usepackage{graphicx}
\usepackage[dvips]{epsfig}

\usepackage{fancyhdr}
\usepackage{times}

\usepackage{enumerate}
\usepackage{titlesec}
\usepackage{mathrsfs}
\usepackage{stmaryrd}

\def\Ccal{\mathcal{C}}
\def\Acal{\mathcal{A}}
\def\Ncal{\mathcal{N}}
\def\Scal{\mathcal{S}}
\def\Mcal{\mathcal{M}}
\def\Rcal{\mathcal{R}}
\def\Hcal{\mathcal{H}}
\def\Kcal{\mathcal{K}}
\def\Bcal{\mathcal{B}}
\def\Dcal{\mathcal{D}}

\def\Tcal{\mathcal{T}}
\def\Ucal{\mathcal{U}}
\def\Vcal{\mathcal{V}}
\def\Jcal{\mathcal{J}}
\def\Wcal{\mathcal{W}}

\def\Cscr{\mathscr{C}}

\def\Pscr{\mathscr{P}}

\def\Jscr{\mathscr{J}}

\def\c{\mathbb{C}}
\def\r{\mathbb{R}}
\def\n{\mathbb{N}}
\def\z{\mathbb{Z}}

\def\s{\mathbb{S}}

\def\d{\mathbb{D}}

\def\Agot{\mathfrak{A}}
\def\Lgot{\mathfrak{L}}

\def\Hgot{\mathfrak{H}}

\def\dgot{\mathfrak{d}}

\def\lP{\llparenthesis}
\def\rP{\rrparenthesis}

\def\dist{\mathrm{dist}}
\def\Fr{\mathrm{Fr}}

\newcommand\wt{\widetilde}
\newcommand\wh{\widehat}

\titleformat{\section}%[display]
{\filcenter\bfseries\large} {\thesection{.}}{0.2cm}{}%[$\vspace*{-1.0cm}$]
\titleformat{\subsection}[runin]
{\bfseries} {\thesubsection{.}}{0.15cm}{}[.]
\titleformat{\subsubsection}[runin]
{\em}{\thesubsubsection{.}}{0.15cm}{}[.]
\usepackage[up,bf]{caption}

\newtheorem{thm}{Theorem}[section]

\newtheorem{claim}[thm]{Claim}
\newtheorem{lem}[thm]{Lemma}
\newtheorem{cor}[thm]{Corollary}

\newtheorem{question}[thm]{Question}

\theoremstyle{definition}
\newtheorem{defin}[thm]{Definition}
\newtheorem{rem}[thm]{Remark}
\numberwithin{equation}{section}
\numberwithin{figure}{section}

\usepackage{color}

\begin{document}

%% CABECERAS
\fancyhead[CO]{Complete bounded embedded complex curves in $\c^2$} % En las páginas impares, parte izquierda del encabezado, aparecerá el nombre de capítulo
\fancyhead[CE]{A. Alarc\'{o}n$\,$  and$\,$ F.J. L\'{o}pez} % En las páginas pares, parte derecha del encabezado, aparecerá el nombre de sección
\fancyhead[RO,LE]{\thepage} % Números de página en las esquinas de los encabezados

\thispagestyle{empty}

\thispagestyle{empty}

%% Title
\vspace*{1cm}
\begin{center}
{\bf\LARGE Complete bounded embedded complex curves in $\c^2$}

\vspace*{0.5cm}

%% Authors
{\large\bf Antonio Alarc\'{o}n$\;$ and$\;$ Francisco J.\ L\'{o}pez}

%\rule[1mm]{20mm}{0.1mm} $\bullet$ \rule[1mm]{20mm}{0.1mm}
\end{center}

%% Addresses and finantial support
\footnote[0]{\vspace*{-0.4cm}

\noindent A.\ Alarc\'{o}n, F.\ J.\ L\'{o}pez

\noindent Departamento de Geometr\'{\i}a y Topolog\'{\i}a e Instituto de Matem\'aticas (IEMath-GR), Universidad de Granada, Campus de Fuentenueva s/n, E--18071 Granada, Spain

\noindent e-mail: {\tt alarcon@ugr.es}, {\tt fjlopez@ugr.es}
}

%% Abstract, keywords and MSC
\vspace*{0.25cm}

\begin{quote}
{\small
\noindent {\bf Abstract}\hspace*{0.1cm} We prove that any  convex domain of $\c^2$ carries properly embedded complete complex curves. In particular, we give the first examples of  complete bounded embedded complex curves in $\c^2$.

\vspace*{0.1cm}

\noindent{\bf Keywords}\hspace*{0.1cm} Riemann surfaces, complex curves, complete holomorphic embeddings.

\vspace*{0.1cm}

\noindent{\bf Mathematics Subject Classification (2010)}\hspace*{0.1cm} 32C22, 32H02, 32B15.
}
\end{quote}

\vspace*{0.25cm}

%%%%%%
%%%%%%
%%%%%%

\section{Introduction}\label{sec:intro}
Let $M^k$ be a $k$-dimensional connected complex manifold, $k\in \n.$ A holomorphic immersion  $X\colon M^k\to \c^n,$ $n\geq k,$ is said to be {\em complete} if the pull back  $X^*g$ of the Euclidean metric $g$ on $\c^n$ is a complete Riemannian metric on $M^k$. This is equivalent to that $X\circ\gamma$ has infinite Euclidean length for any  {\em divergent} arc $\gamma$ in $M^k$. (Given a non-compact topological space $W$, an arc $\gamma\colon [0,1)\to W$ is said to be divergent if $\gamma(t)$ leaves any compact subset of $W$ when $t\to 1$.) 

An immersion  $X\colon M^k\to \c^n$ is said to be an {\em embedding} if $X\colon M^k \to X(M^k)$  is a homeomorphism. In this case $X(M^k)$ is said an embedded submanifold of $\c^n.$ If $\Omega\subset \c^n$ is a domain, a map $X\colon M^k \to \Omega$ is said to be {\em proper} if $X^{-1}(K)$ is compact for any compact set $K\subset \Omega.$ Proper injective immersions $M^k \to \Omega$ are embeddings.

In 1977, Yang  \cite{Yang1,Yang2}  proposed the question of whether there exist complete  holomorphic embeddings $M^k \to \c^n$, $1\leq k<n$, with bounded image. The first affirmative answer was given two years later by Jones \cite{Jones} for $k=1$ and $n\geq 3$. Only recently, Alarc\'on and Forstneri\v c \cite{AF-1}, as application of Jones' result, have provided examples for any $k\in\n$  and $n\geq 3k$. The problem remained open in the lowest dimensional case: complex curves in $\c^2$ (see \cite[Question 1]{AF-1}). This particular case is especially interesting for topological and analytical reasons that will be more apparent later in this introduction.

The aim of this paper is to fill this gap, proving considerably more:

\begin{thm}\label{th:intro1}
Any  convex domain $\Bcal \subset \c^2$ carries complete properly embedded complex curves.
\end{thm}

The topology of the curves in Theorem \ref{th:intro1} is not controlled; see Question \ref{qu:ex} below. The thesis of Theorem \ref{th:intro1} is obvious when $\Bcal=\Omega \times \c$, where $\Omega\subset\c$ is a convex domain (the flat curve $\{p\}\times \c$, $p\in \Omega$, is complete and properly embedded in $\Omega \times \c$). Further, complete holomorphic graphs over $\Omega$ were constructed in \cite{Isa,AFL-1,AFL-2}. Regarding the case $\Bcal=\c^2$, Bell and Narasimhan \cite{BellNarasimhan}  conjectured that any open Riemann surface can be properly holomorphically embedded in $\c^2$ (obviously, this is possible in no other convex domain of $\c^2$). This classical problem is still open; cf. \cite{ForstnericWold1,ForstnericWold2,Forstneric-book,AL-Narasimhan} and references therein.  Anyway, all the complex curves in these particular instances are far from being bounded.

Following Yang's results \cite{Yang2}, no complete complex hypersurface of $\c^{n}$, $n>1$, has strongly negative holomorphic sectional curvature, and the existence of a complete bounded complex $k$-dimensional submanifold of $\c^n$, $n>k$, implies the existence of such a submanifold of $\c^{2n}$ with  strongly negative holomorphic sectional curvature. Related existence results can be found in \cite{AF-1}. Theorem \ref{th:intro1} has nice consequences regarding these questions:
\begin{cor}\label{cor:iii}
Let $k\in\n$. There exist
\begin{enumerate}[{\rm (i)}]
\item  complete bounded {\em embedded} complex $k$-dimensional submanifolds of $\c^{2k}$, and
\item complete bounded {\em embedded} complex $k$-dimensional submanifolds of $\c^{4k}$ with strongly negative holomorphic sectional curvature.
\end{enumerate}
\end{cor}
\begin{proof}
Let $X\colon \Rcal\to B$ be a complete holomorphic embedding given by Theorem \ref{th:intro1}; where $\Rcal$ is an open Riemann surface and $B\subset \c^2$ is the Euclidean open ball of radius $1/\sqrt{k}$ centered at the origin. Denote by $\Rcal^k=\Rcal\times\ldots\times\Rcal$ the cartesian product of $k$ copies of $\Rcal$ and likewise for $B^k$. Then the map
\[
\varphi\colon \Rcal^k \to B^k\subset \c^{2 k},\quad \varphi(p_1,\ldots,p_k)=(X(p_1),\ldots,X(p_k)),
\]
is a complete bounded holomorphic embedding, proving {\rm (i)}; see \cite[Corollary 1]{AF-1}.

To check {\rm (ii)}, notice that $\varphi(\Rcal^k)\subset B_1,$ where $B_1\subset \c^{2 k}$ is the Euclidean open ball of radius $1$ centered at the origin. Setting $F\colon B_1\to \c^{4k},$ $F(z_1,\ldots,z_{2 k})=(z_1,\ldots,z_{2 k},e^{z_1},\ldots,e^{z_{2 k}})$, the map $F\circ \varphi\colon \Rcal^k \to \c^{4 k}$ proves {\rm (ii)}; see \cite[Sec.\ 1]{Yang2}. 
\end{proof}

An interesting question is whether, given $k\in\n$, the dimensions $2k$ and $4k$ in the above corollary are optimal.

There are many known examples of complete bounded {\em immersed} complex curves in $\c^2$; Jones \cite{Jones} constructed a simply-connected one, Mart\'in, Umehara, and Yamada \cite{MUY2} provided examples with some finite topologies, and Alarc\'on and L\'opez \cite{AL-CY} gave examples of arbitrary topological type. On the other hand, Alarc\'on and Forstneri\v c \cite{AF-1} showed that every bordered Riemann surface is a complete curve in a ball of $\c^2$. Furthermore, the curves in \cite{AL-CY,AF-1} have the extra property of being proper in any given convex domain. However, the  construction of complete bounded {\em embedded}  complex curves in $\c^2$  turns out to be a much more involved problem. 
The main reason why is that (contrarily to what happens in $\c^n,$ $n\geq 3$, where the general position of complex curves is embedded) {\em self-intersections of complex curves in $\c^2$ are stable under deformations}. Nevertheless, there is a simple self-intersection removal method which consists of replacing every normal crossing in a complex curve by an embedded annulus. Unfortunately, this surgery  does not necessarily preserve the length of divergent arcs (hence completeness); indeed, self-intersection points of  immersed complex curves generate {\em shortcuts} in the arising desingularized curves, so divergent arcs of shorter length. 

In order to overcome this difficulty, we have considered a stronger notion of completeness (Def.\ \ref{def:extrinsic}). Given a holomorphic immersion $X\colon M^k\to \c^n$, we denote by ${\rm dist}_{X(M^k)}$ the  (intrinsic) induced Euclidean distance in $X(M^k)$ given by 
\[
{\rm dist}_{X(M^k)}(p,q)=\inf\{\ell(\gamma)\colon \gamma \subset X(M^k) \; \text{rectifiable arc connecting $p$ and $q$}\}
\]
for any $p,$ $q\in X(M^k)$; where $\ell(\cdot)$ means Euclidean length in $\c^n$. If $X$ is injective, the function ${\rm dist}_{X(M^k)}\circ (X,X)\colon M^k\times M^k\to \r$ is the intrinsic distance in $M^k$ induced by $X$; otherwise it is a pseudo-distance. We call ${\rm dist}_{X(M^k)}$  and $(X(M^k),{\rm dist}_{X(M^k)})$ the {\em image distance} and  the {\em image metric space} of  $X\colon M^k\to \c^n$.

\begin{defin}\label{def:extrinsic} 
A holomorphic immersion $X\colon M^k \to \c^n$ is said to be {\em image complete} if $(X(M^k),{\rm dist}_{X(M^k)})$ is a complete metric space (in other words, if every rectifiable divergent arc in $X(M^k)$ has infinite Euclidean length).
\end{defin}
Obviously, image completeness  implies completeness, and both notions are equivalent for injective immersions.
The image distance is very convenient for our purposes since it is preserved by self-intersection removal procedures.  As a matter of fact, the proof of Theorem \ref{th:intro1} is connected with the general existence Theorem \ref{th:intro2} below. As far as the authors' knowledge extends, the followings are the first known examples of image complete bounded immersed complex curves in $\c^2$.

\begin{thm}\label{th:intro2}
Let $S$ be an open orientable smooth surface and let $\Bcal\subset\c^2$ be a convex domain. 

Then there exist a complex structure $\Jcal$ on $S$ and an image complete proper holomorphic immersion $(S,\Jcal)\to \Bcal$.
\end{thm}

Let us say a word on the proof of Theorem \ref{th:intro1}; see the more general Theorem \ref{th:embe} in Sec.\ \ref{sec:theorem}. The proof of the theorem relies on a recursive process involving an approximation result by embedded complex curves in $\c^2$ (Lemma \ref{lem:compilation}), which is the core of the paper. In this lemma we prove that any embedded compact complex curve $\Cscr$ with boundary $b \Cscr$ in the frontier $\Fr \Dcal$  of a regular strictly convex domain $\Dcal$, can be approximated by another {\em embedded} complex curve $\Cscr'$ with $b \Cscr'\subset \Fr \Dcal'$, where  $\Dcal'$ is  any given  larger convex domain. The curve $\Cscr'$  has possibly  higher topological genus than $\Cscr$ and contains a biholomorphic copy of it, roughly speaking 
 $\Cscr\subset \Cscr'$. Furthermore, this procedure can be done  so that  $\Cscr'\setminus\Cscr$ lies in $\Dcal'\setminus\Dcal$ and the intrinsic Euclidean distance in  $\Cscr'$ from $\Cscr$ to $b \Cscr'$  is suitably larger than the distance  between $\Dcal$ and $\Fr\Dcal'$ in $\c^2$. These facts will be the key for obtaining  properness and completeness while preserving boundedness in the proof of Theorem \ref{th:embe}.

In order to prove Lemma \ref{lem:compilation} (see Sec.\ \ref{sec:lemma-proof}), we have introduced some configurations of  slabs in $\c^2$  that we have called {\em tangent nets} (see Subsec.\ \ref{subsec:net}). Given a regular strictly convex domain $\Dcal\Subset\c^2$, a tangent net $\Tcal$ for $\Dcal$ is a tubular neighborhood of a finite collection of (affine) tangent hyperplanes to the frontier $\Fr\Dcal$; see Def.\ \ref{def:net} and Fig.\ \ref{fig:net}. Given another regular strictly convex domain $\Dcal'$, $\Dcal\Subset\Dcal'\Subset\c^2$, we show the existence of tangent nets $\Tcal$ for $\Dcal$ with the property  that any Jordan arc in $\Tcal$  connecting $\Fr\Dcal$ and $\Fr\Dcal'$ has {\em large}  length comparatively to the distance  between $\Dcal$ and $\Fr\Dcal'$ in $\c^2$; see Lemma \ref{lem:net}.  
The second step in the proof of Lemma \ref{lem:compilation} is an approximation result by immersed complex curves along tangent nets (see Lemma \ref{lem:main} in Subsec.\ \ref{subsec:lemma}). It asserts that any immersed compact complex curve $\Sigma$ in $\c^2$ with boundary  $b \Sigma \subset \Fr\Dcal$, can be approximated by another one $\wt\Sigma$ such that $b \wt\Sigma  \subset \Fr\Dcal'$ and  $\wt\Sigma\cap (\overline{\Dcal'}\setminus \Dcal)$ is contained inside a suitable tangent net for $\Dcal$. This allows us to estimate the growth of the {\em image diameter} (according to Def.\ \ref{def:extrinsic}) of $\wt\Sigma$, and conclude that it is large comparatively to the distance between $\Dcal$ and $\Fr\Dcal'$. This represents a clear innovation with respect to previous constructions where only the growth of the {\em intrinsic diameter} could be estimated (cf.\ \cite{Nadirashvili,AL-CY,AF-1} and references therein). 
We conclude the proof of Lemma \ref{lem:compilation} by combining the above two results with a desingularization  result adapted to our setting (see Lemma \ref{lem:embe} in Subsec.\ \ref{subsec:embe}).  To the best of the authors' knowledge, this is the first such application of the surgery technique in the literature. Since this method increases the topology, the complex curves in Theorem \ref{th:intro1} could have infinite genus.

On the other hand, Theorem \ref{th:intro2} follows from a standard recursive application of Lemmas \ref{lem:net} and \ref{lem:main} (see the more precise Theorem \ref{th:extrinsic} in Sec.\ \ref{sec:extrinsic}).  

Since complex curves in $\c^2$ are area-minimizing surfaces in $\r^4$, our results connect with the so-called {\em Calabi-Yau problem for embedded surfaces}. This problem deals with the existence of complete embedded minimal surfaces in bounded domains of $\r^3$. Although it still remains open, it is known that solutions must have either infinite genus or uncountably many ends (see Colding and Minicozzi \cite{ColdingMinicozzi-CY} and Meeks, P\'erez, and Ros \cite{MeeksPerezRos-CY}). On the other hand, 
the construction of embedded complex discs in $\c^2$ is a subject with vast literature; see for instance \cite{Glo,FGS,Forstneric-Glo,Glo2,Glo1}. Thus, in view of Theorem \ref{th:intro1}, one is led to ask:

\begin{question} \label{qu:ex}
Do there exist complete bounded holomorphic embeddings $M\to \c^2$ with $M$ an open Riemann surface of finite topology? What is the answer if $M$ is the  complex unit disc?
\end{question}

Our main tools are  the classical Runge and Mergelyan approximation theorems for holomorphic functions and  basic convex body theory. 

%%%%%%
%%%%%%
%%%%%%

\section{Preliminaries}\label{sec:prelim}

We denote by $\|\cdot\|$, $\langle\cdot,\cdot\rangle$, $\dist(\cdot,\cdot)$, $\ell(\cdot)$, and ${\rm diam}(\cdot)$ the Euclidean norm, inner product, distance, length, and diameter in $\r^n$, $n\in\n$. 
Given two points $p$ and $q$ in $\r^n$, we denote by $[p,q]$ (resp., $]p,q[$) the closed (resp., open) straight segment in $\r^n$ connecting $p$ and $q$. 

In the complex Euclidean space $\c^n\cong\r^{2n}$ we denote by $\lP\cdot,\cdot\rP\colon\c^n\times\c^n\to\c$ the bilinear Hermitian product defined by $\lP (\zeta_1,\ldots,\zeta_n),(\xi_1,\ldots,\xi_n)\rP = \sum_{i=1}^n \zeta_i\overline{\xi}_i$, where $\overline{\,\cdot}$ means complex conjugation. Observe that $\langle\cdot,\cdot\rangle =\Re \lP\cdot,\cdot\rP$. Given $p\in \c^n,$ we denote by $\lP p \rP^\bot=\{q\in \c^n\colon \lP p,q\rP=0\},$ ${\rm span}_\r(p)=\{t p\colon t\in\r\}$, and ${\rm span}_\c(p)=\{\zeta p\colon \zeta\in\c\}$.

%\subsection{Riemann surfaces}

Given an $n$-dimensional topological real manifold $M$ with boundary, we denote by $bM$ the $(n-1)$-dimensional
topological manifold determined by its boundary points. For any subset $A\subset M$, we denote by
$A^\circ$, $\overline{A}$, and $\Fr A=\overline{A}\setminus A^\circ$, the interior, the closure, and the frontier of $A$ in $M$, respectively. Given subsets $A$ and $B$ of $M$, we write $A\Subset B$ if $\overline{A}$ is compact and $\overline{A}\subset B^\circ$. By a {\em domain} in $M$ we mean an open connected subset of $M\setminus bM$. By a {\em region} in $M$ we mean a proper topological
subspace of $M$ being an $n$-dimensional compact manifold with non-empty boundary.

A topological surface $S$ is said to be {\em open} if it is non-compact and $b S=\emptyset$. A domain $\Rcal$ in an open connected Riemann surface $\Ncal$ is said to be a {\em bordered domain} if
$\Rcal\Subset \Ncal$ and $\overline{\Rcal}$ is a region with smooth boundary $b\overline{\Rcal}=\Fr\Rcal$. In this case,  $b\overline{\Rcal}$ consists of finitely many smooth Jordan curves.

Given a compact topological space $K$ and a continuous map $f\colon K\to \r^n,$ we denote by 
\[
\|f\|_{0,K}:= \max_{p\in K} \|f(p)\|
\]
the maximum norm of $f$ on $K.$ The corresponding space of continuous functions $K \to \r^n$ will be endowed with the $\Ccal^0$ topology associated to $\|\cdot\|_{0,K}.$

Let  $\Ncal$ be an open Riemann surface endowed with a  nowhere-vanishing holomorphic $1$-form  $\vartheta_\Ncal$  (such a $1$-form exists by the Gunning-Narasimhan theorem \cite{GunningNarasimhan}). Let $K$ be a compact set in  $\Ncal$. A function $f\colon K\to \c^n$, $n\in\n$, is said be holomorphic if it is the restriction to $K$ of a holomorphic function defined on a domain in $\Ncal$ containing $K$. In such case, we denote by  
\begin{equation}\label{eq:norma1}
\|f\|_{1,K;\vartheta_\Ncal}:=\max_{P\in K} \max \big\{ \|f(P)\|\,,\,\|\frac{df}{\vartheta_\Ncal}(P)\|\big\}
\end{equation}
the $\Ccal^1$ maximum norm of $f$ on $K$ (with respect to $\vartheta_\Ncal$). If there is no place for ambiguity, we write $\|f\|_{1,K}$ instead of $\|f\|_{1,K;\vartheta_\Ncal}$.  The space of holomorphic functions $K\to \c^n$ will be endowed with the $\Ccal^1$ topology associated to the norm $\|\cdot\|_{1,K;\vartheta_\Ncal}$, which does not depend on the choice of $\vartheta_\Ncal$.

Given a holomorphic immersion $f\colon K\to\c^n$, a point $w\in f(K)$ is said to be a {\em double point of $f$} (or of $f(K)$) if $f^{-1}(w)$ contains more than one point. A double point $w\in f(K)$ is said to be a {\em normal crossing} if $f^{-1}(w)$ consists of precisely two points, $P$ and $Q$, and $df_P(T_P\Ncal)$ and $df_Q(T_Q\Ncal)$ are transverse.

\begin{rem}\label{rem:position}
It is well known that any holomorphic function $K\to\c^n$, $n\geq 3$, can be approximated  in the $\Ccal^1$ topology on $K$ by holomorphic embeddings. 

However, this is no longer true in the lowest dimensional case; double points of an immersed complex curve in $\c^2$ are stable under deformations. Anyway, any holomorphic function $K\to\c^2$ can be approximated  in the $\Ccal^1$ topology on $K$ by holomorphic immersions all whose double points are normal crossings. We call this property {\em the general position argument}.
\end{rem}
Throughout this paper we will deal with  regular  convex domains $\Dcal\Subset \c^2$, bordered  domains $\Rcal \Subset \Ncal$, and holomorphic immersions $X:\overline{\Rcal} \to \c^2$ with $X(\overline{\Rcal})\subset \Dcal$.  In this setting, it  is interesting to notice that:
\begin{rem} \label{re:trans}
If $X(b \overline{\Rcal})\subset \Fr \Dcal$ then $X(\overline{\Rcal})$ and $\Fr \Dcal$ meet transversally. 
\end{rem}
Indeed, assume for a moment that $X(\overline{\Rcal})$ and $\Fr \Dcal$ meet tangentially at $p:=X(P)$, $P\in b\overline{\Rcal}$. By basic theory of harmonic functions, there exists a sufficiently small neighborhood $U$ of $P$ in $ \Mcal$ such that $\alpha:=X^{-1}(p+T_p \Fr \Dcal)\cap U$ consists of a system of at least two analytical arcs  intersecting equiangularly at $P$. Furthermore,  contiguous components of $U\setminus\alpha$ lie in opposite sides of $p+T_p \Fr \Dcal$. On the other hand, since $\overline{\Rcal}$ has smooth boundary and $X(b \overline{\Rcal})\subset \Fr \Dcal$ then $X(U\cap \overline{\Rcal})\subset \Dcal$, and so, $X(U\cap \overline{\Rcal})$ must lie at one side of $p+T_p \Fr \Dcal$, a contradiction.

A  compact (in most cases arcwise-connected) subset $K$ of an open Riemann surface $\Ncal$ is said to be {\em Runge} if $\Ncal\setminus K$ has no relatively compact connected components in $\Ncal$; equivalently, if the inclusion map $i\colon K\hookrightarrow \Ncal$ induces a group monomorphism on homology $i_*\colon \Hcal_1(K,\z) \to \Hcal_1(\Ncal,\z).$  In this case we consider $\Hcal_1(K,\z) \subset \Hcal_1(\Ncal,\z)$ via this monomorphism. Two Runge compact sets  $K_1$ and $K_2$ of $\Ncal$ are said to be {\em (homeomorphically)  isotopic} if there exists  a homeomorphism $\eta\colon K_1 \to K_2$ such that the induced group morphism on homology, namely $\eta_*$, equals ${\rm Id}_{\Hcal_1(K_1,\z)}$. Such an $\eta$ is said to be an {\em isotopical homeomorphism}. Two Runge regions $K_1$ and $K_2$ of $\Ncal$ are (homeomorphically)  isotopic  if and only if $\Hcal_1(K_1,\z)=\Hcal_1(K_2,\z).$

%%%%%%
%%%%%%
%%%%%%

\subsection{Convex Domains}

A convex domain $\Dcal\subset\r^n$, $\Dcal \neq \r^n$, $n\geq 2$, is said to be {\em regular} (resp., {\em analytic}) if its frontier $\Fr\Dcal=\overline{\Dcal}\setminus\Dcal$ is a regular (resp.,  analytic) hypersurface of $\r^n$.

Let $\Dcal$ be a regular convex domain of $\r^n$, $\Dcal \neq \r^n$, $n\geq 2$.

For any $p\in\Fr\Dcal$ we denote by $T_p \Fr\Dcal$ the tangent space to $\Fr \Dcal$ at $p$. Recall that  ${\Dcal} \cap (p+T_p \, \Fr \Dcal)=\emptyset$ for all $p \in \Fr \Dcal$.

We denote by $\nu_\Dcal\colon \Fr \Dcal \rightarrow \s^{n-1}$ the
outward pointing unit normal of $\Fr \Dcal$. 
For any $p\in\Fr \Dcal$ and $v \in (T_p\, \Fr \Dcal)\cap \,\s^{n-1}$, we denote by  $\kappa_\Dcal (p,v)$ the  normal curvature at $p$ in the direction of $v$   with respect to $-\nu_\Dcal$; obviously $\kappa_\Dcal(p,v)\geq 0$ since $\Dcal$ is convex. Let $\kappa(p)\geq 0$ be
the maximum of the principal curvatures of $\Fr \Dcal$ at $p$ with respect to $-\nu_\Dcal$, and set
\begin{equation}\label{eq:kD}
\kappa (\Dcal):= \sup \{\kappa(p)\colon p\in\Fr \Dcal\} \geq 0.%\cup\{+\infty\}.
\end{equation}

The domain $\Dcal$ is said to be {\em strictly convex} if $\kappa_\Dcal (p,v)>0$ for all $p \in \Fr \Dcal$ and $v \in (T_p \, \Fr \Dcal)\cap \s^{n-1}$. In this case,  $\overline{\Dcal} \cap (p+T_p \, \Fr \Dcal)=\{p\}$ for all $p \in \Fr \Dcal$. 
If $\Dcal$ is bounded (i.e., $\Dcal\Subset\r^n$) and strictly convex, then $0<\kappa(\Dcal)<+\infty$.

Assume that $\Dcal$ is bounded  and strictly convex. For any $t >-1/\kappa(\Dcal)$ we denote by $\Dcal_t$ the bounded regular strictly convex domain in $\r^n$ with frontier $\Fr \Dcal_t=\{p+ t\, \nu_\Dcal(p) \colon p \in \Fr \Dcal\}$; that is, the {\em parallel convex domain} to $\Dcal$ at (oriented) distance $t$. Observe that $\Dcal=\Dcal_0$ and $\Dcal_{t_1}\Subset \Dcal_{t_2}$ if $t_1<t_2$. 

%Set $\Dcal_{-1/\kappa(\Dcal)}:=\cap_{t>-1/\kappa(\Dcal)} D_t$ and denote by
%\[
%\pi_\Dcal\colon \r^n\setminus\Dcal_{-1/\kappa(\Dcal)}\to \Fr \Dcal,\enskip \pi_\Dcal(p+t\nu_\Dcal(p))=p,
%\]
%the normal projection. We denote by  $\nu_\Dcal$ the {\em extended normal map} $\nu_\Dcal\circ \pi_\Dcal\colon \r^n\setminus\Dcal_{-1/\kappa(\Dcal)} \to \s^{n-1}.$

For any couple of compact subsets $K$ and $O$ in $\r^n$,
the {\em Hausdorff distance} between $K$ and $O$ is given by
\[
\dgot^{\rm H}(K,O):=\max\Big\{ \sup_{x\in K} \inf_{y\in O} \|x-y\|\;,\;
\sup_{y\in K} \inf_{x\in O} \|x-y\|\Big\}.
\]

A sequence $\{K^j\}_{j \in \n}$ of (possibly unbounded) closed subsets of $\r^n$ is said to converge in the {\em Hausdorff topology} to a closed subset $K^0$ of $\r^n$ if
$\{K^j\cap B\}_{j \in \n}\to K^0 \cap B$ in the Hausdorff distance for any closed Euclidean ball $B \subset \r^n.$ If $K^j\Subset K^{j+1} \subset K^0$ for all $j\in\n$ and $\{K^j\}_{j \in \n} \to K^0$ in the Hausdorff topology, then we write $\{K^j\}_{j \in \n} \nearrow K^0.$ %Likewise we write $\{K^j\}_{j \in \n} \searrow K^0$ provided that $K^0\subset K^{j+1}\Subset K^j$ for all $j\in\n$ and $\{K^j\}_{j \in \n} \to K^0$ in the Hausdorff topology.

%The following approximation result follows from classical Minkowski's Theorem.

\begin{thm}[\cite{Minkowski,MeeksYau}]\label{th:mink}
Let $\Bcal\subset\r^n$ be a (possibly neither bounded nor regular) convex domain. 
Then there exists a sequence $\{\Dcal^j\}_{j\in\n}$ of   bounded strictly convex analytic domains in $\r^n$ with $\{\overline{\Dcal^j}\}_{j \in \n} \nearrow \overline{\Bcal}.$
%If in addition $C$ is bounded, then there exists a sequence $\{\Dcal^j\}_{j\in\n}$ of   bounded strictly convex analytic domains in $\r^n$ with $\{\overline{\Dcal^j}\}_{j \in \n} \searrow \overline{C}.$ 
\end{thm}

The following {\em distance type}  function for convex domains will play a fundamental role throughout this paper.
\begin{defin} \label{def:d}
Let $\Dcal$ and $\Dcal'$ be  bounded regular strictly convex domains in $\r^n$ ($n\geq 2$), $\Dcal \Subset \Dcal'.$ We denote by
\[
{\bf d}(\Dcal,\Fr\Dcal'):=\Big(\dist(\Dcal,\Fr\Dcal')+\frac{1}{\kappa(\Dcal)}\Big)\sqrt{\frac{\dist(\Dcal,\Fr\Dcal')}{\dist(\Dcal,\Fr\Dcal')+2/\kappa(\Dcal)}}
\]
(see \eqref{eq:kD}).
\end{defin}

\begin{rem} \label{re:compara}
Observe that ${\bf d}(\Dcal,\Fr\Dcal')>\dist(\Dcal,\Fr\Dcal')$. Furthermore,  ${\bf d}(\Dcal,\cdot)$ and $\sqrt{\dist(\Dcal,\cdot)}$ are infinitesimally comparable in the  sense that $\lim_{n\to \infty} \frac{\sqrt{{\rm dist}(\Dcal,\Fr\Dcal^n)}}{{\bf d}(\Dcal,\Fr\Dcal^n)}=\sqrt{2\kappa(\Dcal)}>0$ for any sequence $\{\Dcal^n\}_{n\in \n}$ of bounded regular strictly convex domains such that $\Dcal \Subset \Dcal^n$ $\forall n\in\n$ and $\{\overline{\Dcal^n}\}_{n\in \n} \to \overline{\Dcal}$ in the Hausdorff topology.
\end{rem} 

Lemma \ref{lem:sucprop} below  will simplify the exposition of the proof of our main results. Its proof relies on the above Remark \ref{re:compara}.

\begin{defin} \label{def:proper}
Let $\Bcal$ be a (possibly neither bounded nor regular) convex domain  in $\r^n.$ A sequence $\{\Dcal^k\}_{k \in \n}$ of   convex domains in $\r^n$ is said to
be ${\bf d}$-{\em proper in $\Bcal$} if $\Dcal^k$ is bounded, regular, and strictly  convex for all $k\in\n$,  $\{\overline{\Dcal^k}\}_{k \in \n} \nearrow \overline{\Bcal}$ in the Hausdorff topology, and 
\[
\sum_{k \in \n} {\bf d}(\Dcal^k,\Fr \Dcal^{k+1})=+\infty.
\]
\end{defin}

\begin{lem}\label{lem:sucprop}
Any convex domain in $\r^n$ admits a ${\bf d}$-proper sequence of convex domains.
\end{lem}
\begin{proof}
Let $\Bcal$ be a convex domain  in $\r^n.$ Let $\{C^j\}_{j\in\n}$ be a sequence of bounded strictly convex analytic domains in $\r^n$ with $\{\overline{C^j}\}_{j \in \n} \nearrow \overline{\Bcal}$; cf. Theorem \ref{th:mink}. For the sake of simplicity write $d_j:={\rm dist}(C^j,\Fr C^{j+1})$ and $\kappa_j:=\kappa(C^j)$ for all $j\in\n$.

For each $j \in \n$ choose $m_j\in \n$ large enough so that 
\begin{equation}\label{eq:m_j}
\sum_{a=1}^{m_j}  \frac1{a}\geq \sqrt{\frac{6d_j\kappa_j^2+2\pi^2\kappa_j}{6d_j}}.
\end{equation}
Denote by $d_{a,j}=d_j\frac{6}{\pi^2}\sum_{h=1}^a 1/h^2$, and notice that $d_{a,j}<d_j$; take into account that $\sum_{h=1}^\infty 1/h^2=\pi^2/6$. Set  $C^{0,j}:=C^j$ and $C^{a,j}:=(C^j)_{d_{a,j}}$, for all $a=1,\ldots,m_j,$
the outer parallel convex domain to $C^j$ at distance $d_{a,j}$.
Observe that $C^{a,j}$ is analytic and strictly convex, 
\begin{equation}\label{eq:nearrow}
C^j\Subset C^{a,j} \Subset C^{a+1,j} \Subset C^{j+1},
\end{equation} 
\begin{equation}\label{eq:daj}
{\rm dist}(C^{a,j},\Fr C^{a+1,j})=d_{a+1,j}-d_{a,j}=6d_j/(\pi (a+1))^2,
\end{equation}
and 
\begin{equation}\label{eq:kcaj}
\kappa(C^{a,j})=\kappa_j/(1+d_{a,j} \kappa_j)\leq \kappa_j\enskip \text{for all $j\in\n$ and $a\in\{0,\ldots,m_j-1\}$.}
\end{equation}
Set
\[
f\colon ]0,+\infty[\times ]0,+\infty[\to]0,+\infty[,\quad f(d,\kappa)=(d+1/\kappa)\sqrt{\frac{d}{d+2/\kappa}},
\]
and note that %$f(\cdot,\kappa)$ is increasing for all $\kappa>0$, 
$f(d,\cdot)$ is decreasing for all $d>0$ and $f(6d_j/(\pi (a+1))^2,\kappa(C^{a,j}))={\bf d}(C^{a,j},\Fr C^{a+1,j})$ for all $j\in\n$ and $a\in\{1,\ldots,m_j-1\}$; see \eqref{eq:daj}.
 Therefore, 
\begin{eqnarray*}
\sum_{a=0}^{m_j-1} {\bf d}(C^{a,j},\Fr C^{a+1,j}) & = &  \sum_{a=0}^{m_j-1} f\big( 6d_j/(\pi (a+1))^2 , \kappa(C^{a,j})\big)\\
& \stackrel{\eqref{eq:kcaj}}{\geq} &  \sum_{a=0}^{m_j-1} f\big( 6d_j/(\pi (a+1))^2 , \kappa_j\big)\\
 & > & \sqrt{\frac{6d_j}{6d_j\kappa_j^2+2\pi^2\kappa_j}}\Big(\sum_{a=0}^{m_j-1}  \frac1{a+1}\Big)\; \stackrel{\eqref{eq:m_j}}{\geq} \; 1.
\end{eqnarray*}

Let $\{\Dcal^k\}_{k \in \n}$ denote the enumeration of $ \{C^{a,j}\colon j \in \n,\, a\in\{0,\ldots,m_j\}\}$ such that $\Dcal^k \Subset \Dcal^{k+1}$ for all $k\in\n$; see  \eqref{eq:nearrow}.  Then 
\[
\sum_{k \in \n} {\bf d}(\Dcal^k,\Fr \Dcal^{k+1}) \geq \sum_{j \in \n} \Big( \sum_{a=0}^{m_j-1} {\bf d}(C^{a,j},\Fr C^{a+1,j})\Big)\geq \sum_{j\in\n} 1=+\infty.
\] 
This property and the fact that $\{\overline{C^j}\}_{j\in\n}\nearrow \overline{\Bcal}$  imply that the sequence $\{\Dcal^k\}_{k \in \n}$ is ${\bf d}$-proper in $\Bcal$. This proves the lemma.
\end{proof}

\section{Complete properly embedded complex curves in convex domains of $\c^2$}\label{sec:theorem}

In this section we prove the main result of this paper; Theorem \ref{th:intro1}. It will be a particular instance of the following more precise result.

\begin{thm}\label{th:embe}
Let $\Bcal$ be a (possibly neither bounded nor regular) convex domain in $\c^2$. Let $\Dcal\Subset\Bcal$ be a strictly convex bounded regular domain. Let $\Ncal$ be an open Riemann surface equipped with a nowhere-vanishing holomorphic $1$-form  $\vartheta_\Ncal$, and let $\Rcal$ be a bordered domain in $\Ncal$.

Then, for any $\varepsilon\in]0,\min \{\dist(\overline{\Dcal},\Fr\Bcal),1/\kappa(\Dcal)\}[$ and any holomorphic embedding $X\colon \overline{\Rcal}\to\c^2$ such that
\begin{equation}\label{eq:th-embe}
X(b \overline{\Rcal})\subset \Fr \Dcal,
\end{equation}
there exist an open Riemann surface $\Mcal$ (possibly of infinite topological genus) and a complete holomorphic embedding $Y\colon\Mcal \to \c^2$ enjoying the following properties:
\begin{enumerate}[\rm (i)]
\item $\overline{\Rcal} \subset \Mcal$.
\item $\|Y-X\|_{1,\overline{\Rcal};\vartheta_\Ncal}<\varepsilon$ (see \eqref{eq:norma1}).
\item $Y(\Mcal)\subset\Bcal$ and $Y\colon\Mcal\to\Bcal$ is a proper map.
\item $Y(\Mcal\setminus \overline{\Rcal})\subset \Bcal\setminus \Dcal_{-\varepsilon}$.
\end{enumerate}
\end{thm}

The proof of Theorem \ref{th:embe} follows from a recursive process involving the following approximation result by embedded complex curves. 

\begin{lem}[Approximation by embedded complex curves]\label{lem:compilation}
Let $\Dcal$ and $\Dcal'$ be bounded regular strictly convex domains in $\c^2$, $\Dcal\Subset\Dcal'$. Let $\Ncal$ be an open Riemann surface equipped with a nowhere-vanishing holomorphic $1$-form $\vartheta_\Ncal$ and let $\Ucal$ be a bordered domain in $\Ncal$.

Then, for any $\epsilon\in]0,\min\{{\rm dist}(\overline{\Dcal},\Fr\Dcal'),1/\kappa(\Dcal)\}[$ and any holomorphic embedding $X\colon \overline{\Ucal}\to\c^2$ such that
\begin{equation}\label{eq:compilation1}
X(b \overline{\Ucal})\subset \Fr \Dcal,
\end{equation}
there exist an open Riemann surface $\Ncal'$, a bordered domain $\Ucal'\Subset\Ncal'$, and a holomorphic embedding $X'\colon \overline{\Ucal'}\to\c^2$ enjoying the following properties:
\begin{enumerate}[\it i)]
\item $\overline{\Ucal}\subset \Ucal'$.
\item $\|X'-X\|_{1,\overline{\Ucal};\vartheta_\Ncal}<\epsilon$.
\item $X'(b\overline{\Ucal'})\subset \Fr\Dcal'$.
\item $X'(\overline{\Ucal'}\setminus\Ucal)\cap \overline{\Dcal}_{-\epsilon}=\emptyset$.
\item $\ell(X'(\gamma))>{\bf d}(\Dcal,\Fr\Dcal')-\epsilon$ for any Jordan arc $\gamma$ in $\overline{\Ucal'}$ connecting $b\overline{\Ucal}$ and $b\overline{\Ucal'}$.
\end{enumerate}
\end{lem}
Roughly speaking, this lemma ensures that any embedded compact complex curve $X:\overline{\Ucal} \to \c^2$ with boundary in the frontier of a regular strictly convex domain $\Dcal\Subset\c^2$, can be approximated by another {\em embedded} complex curve   $X' \colon \overline{\Ucal'} \to \c^2$ with boundary in the frontier of a larger convex domain $\Dcal'$. This can be done  so that  $X'(\overline{\Ucal'}\setminus \Ucal)$ lies outside $\Dcal$ and the intrinsic diameter of $X'(\overline{\Ucal'})$ exceeds in ${\bf d}(\Dcal,\Fr\Dcal')$ the one of $X(\Ucal)$; see Def.\ \ref{def:d}. These facts will be the key for obtaining  properness and completeness while preserving boundedness in the proof of Theorem \ref{th:embe}. We point out that $\Ucal'$ has possibly higher topological genus than $\Ucal$.

Lemma \ref{lem:compilation} will be proved later in Sec.\ \ref{sec:lemma-proof}; see in particular Subsec.\ \ref{subsec:proofcompi}. We are now ready to prove our main result.

\begin{proof}[Proof of Theorem \ref{th:embe}]
Denote by $\Dcal^0:=\Dcal$ and let $\{\Dcal^n\}_{n\in\n}$ be a ${\bf d}$-proper sequence of convex domains in $\Bcal$ with $\Dcal^0\Subset\Dcal^1$; see Def.\ \ref{def:proper} and Lemma \ref{lem:sucprop}. Call $\Ncal_0=\Ncal,$ $\vartheta_0=\vartheta_\Ncal$, $\Ucal_0=\Rcal$, and $X_0=X$. Fix any $\epsilon_0\in]0,\varepsilon/2[$. 

Let us recursively construct a sequence $\{\Xi_n=(\Ncal_n, \vartheta_n, \Ucal_n, X_n, \epsilon_n)\}_{n\in\n}$; where
\begin{itemize}
\item $\Ncal_n$ is an open Riemann surface,
\item $\vartheta_n$ is a nowhere-vanishing holomorphic $1$-form on $\Ncal_n$,
\item $\Ucal_n\Subset\Ncal_n$ is a bordered domain,
\item $X_n\colon\overline{\Ucal}_n\to\c^2$ is a holomorphic embedding, and
\item $\epsilon_n\in]0,\min\{{\rm dist}(\Dcal^{n-1},\Fr\Dcal^{n}),1/\kappa(\Dcal^{n-1})\}[$,
\end{itemize}
such that the following properties are satisfied for all $n\in\n$:

\begin{enumerate}[\rm (A$_n$)]
\item $\overline{\Ucal}_{n-1}\subset \Ucal_n$ (in particular, the closure of $\Ucal_{n-1}$ in $\Ncal_{n-1}$ agrees with the one in $\Ncal_n$).
\item $\min_{\overline{\Ucal}_{n-1}}|\vartheta_{n-1}/\vartheta_n|>1$.
\item $\epsilon_n$ verifies that
\begin{enumerate}[\rm (C.1$_n$)]
\item $\epsilon_n<\epsilon_{n-1}/2<\varepsilon/2^{n+1}$ and 
\item every holomorphic function $F\colon \overline{\Ucal}_{n-1}\to\c^2$ with 
$\|F-X_{n-1}\|_{1,\overline{\Ucal}_{n-1};\vartheta_{n-1}}<2\epsilon_n$
is an embedding on $\overline{\Ucal}_{n-1}$. 
\end{enumerate}
\item $\|X_n-X_{n-1}\|_{1,\overline{\Ucal}_{n-1};\vartheta_{n-1}}<\epsilon_n$.
\item $X_n(b\overline{\Ucal}_n)\subset \Fr\Dcal^n$; hence   $X_n(\overline{\Ucal}_n)$ and $\Fr\Dcal^n$  meet transversally (see Remark \ref{re:trans}).
\item $X_n(\overline{\Ucal}_a\setminus{\Ucal}_{a-1})\cap \overline{\Dcal^{a-1}_{-\epsilon_{a}}}=\emptyset$ for all $a\in\{1,\ldots,n\}$.
\item $\ell(X_n(\gamma))>{\bf d}(\Dcal^{a-1},\Fr\Dcal^a)-\epsilon_a$ for any Jordan arc $\gamma$ in $\overline{\Ucal}_a$ connecting $b\overline{\Ucal}_{a-1}$ and $b\overline{\Ucal}_{a}$, for all $a\in\{1,\ldots,n\}$.
\end{enumerate}

The basis of the induction is given by setting $\Xi_0=(\Ncal_0, \vartheta_0, \Ucal_0, X_0, \epsilon_0)$. Remark \ref{re:trans}  gives that $X_0(\overline{\Ucal}_0)$ and $\Fr\Dcal^0$ meet transversally, proving  {\rm (E$_0$)}. Properties {\rm ($j_0$)}, $j\neq {\rm E},$ are empty.

For the inductive step, let $n\in\n$, assume that we have already constructed $\Xi_m$ for all $m\in\{0,\ldots,n-1\}$, and let us construct $\Xi_n$.

Let $\epsilon_n$ be a real number in $]0,\min\{{\rm dist}(\Dcal^{n-1},\Fr\Dcal^{n}),1/\kappa(\Dcal^{n-1})\}[$ and satisfying {\rm (C$_n$)} to be specified later. By {\rm (E$_{n-1}$)}, Lemma \ref{lem:compilation} applies to the data
\[
(\Dcal,\Dcal',\Ncal,\vartheta_\Ncal,\Ucal,\epsilon,X)=(\Dcal^{n-1},\Dcal^n,\Ncal^{n-1},\vartheta_{n-1},\Ucal_{n-1},\epsilon_n,X_{n-1})
\]
furnishing an open Riemann surface $\Ncal^n$, a bordered domain $\Ucal_n\Subset\Ncal^n$, and a holomorphic embedding $X_n\colon \overline{\Ucal}_n\to\c^2$ satisfying {\rm (A$_n$)}, {\rm (D$_n$)}, {\rm (E$_n$)}, and properties {\rm (F$_n$)} and {\rm (G$_n$)} for $a=n$. Further, {\rm (F$_n$)} and {\rm (G$_n$)} for $a\in\{1,\ldots,n-1\}$ are ensured from {\rm (F$_{n-1}$)}, {\rm (G$_{n-1}$)}, and {\rm (D$_n$)}, provided that $\epsilon_n$ is chosen small enough. Up to taking any nowhere-vanishing holomorphic $1$-form $\vartheta_n$ in $\Ncal_n$ satisfying {\rm (B$_n$)}, this closes the induction and concludes the construction of the sequence $\{\Xi_n\}_{n\in\n}$.

Denote by $\Mcal$ the open Riemman surface $\cup_{n\in\n}\Ucal_n$; observe that properties {\rm (A$_n$)}, $n\in\n$, imply Theorem \ref{th:embe}-{\rm (i)}. The sequence $\{X_n\colon\overline{\Ucal}_n\to\c^2\}_{n\in\n}$ converges uniformly on compact sets of $\Mcal$ to a holomorphic map 
\[
Y\colon\Mcal\to\c^2;
\]
just observe that properties {\rm (B$_n$)}, {\rm (C.1$_n$)}, and {\rm (D$_n$)} guarantee that 
\begin{equation}\label{eq:co}
\|X_n-X_{n-1}\|_{1,\overline{\Ucal}_k;\vartheta_k}<\epsilon_n<\varepsilon/2^{n+1}\quad\text{for any $k<n$}.
\end{equation}

Let us show that the map $Y$ satisfies all the requirements in the theorem.

\noindent$\bullet$ $Y$ is an injective immersion. Indeed, for every $k\in\n$, \eqref{eq:co} and {\rm (C.1$_n$)}, $n>k$, give that
\begin{equation}\label{eq:co'}
\|Y-X_k\|_{1,\overline{\Ucal}_k;\vartheta_k}\leq \sum_{n>k}\|X_n-X_{n-1}\|_{1,\overline{\Ucal}_k;\vartheta_k} < \sum_{n>k}\epsilon_n < 2\epsilon_{k+1}<\epsilon_k.
\end{equation}
This and  {\rm (C.2$_n$)} ensure that $Y|_{\overline{\Ucal}_k}\colon \overline{\Ucal}_k\to\c^2$ is an embedding for all $k\in\n$, hence $Y$ is an injective immersion as claimed.

\noindent$\bullet$ $Y$ is complete. Indeed, from  {\rm (G$_n$)}, $n\in \n,$ and taking limits as $n\to\infty$,  we infer that $\ell(Y(\gamma))\geq {\bf d}(\Dcal^{n-1},\Fr\Dcal^n)-\epsilon_n$ for any Jordan arc $\gamma$ in $\overline{\Ucal}_n$ connecting $b\overline{\Ucal}_{n-1}$ and $b\overline{\Ucal}_n$, for all $n\in\n$. Therefore, if $\alpha\in\Mcal$  is a divergent arc in $\Mcal$  with initial point in $\Rcal=\Ucal_0$, one infers that 
 $\ell(Y(\alpha))\geq \sum_{n\in\n} ({\bf d}(\Dcal^{n-1},\Fr\Dcal^n)-\epsilon_n)=+\infty$; take into account that $\{\overline{\Ucal}_n\}_{n\in \n}$ is an exhaustion by compact sets of $\Mcal,$ the series $\sum_{n\in\n}\epsilon_n$ is convergent (see {\rm (C.1$_n$)}), and $\sum_{n\in\n}{\bf d}(\Dcal^{n-1},\Fr\Dcal^n)$ is divergent (recall that $\{\Dcal^n\}_{n\in \n}$ is ${\bf d}$-proper in $\Bcal$; see Def.\ \ref{def:proper}). This ensures the completeness of $Y$.

\noindent$\bullet$ Item {\rm (ii)} is given by \eqref{eq:co'} for $k=0$ (recall that $\epsilon_0<\varepsilon$).

\noindent$\bullet$ $Y(\Mcal)\subset\Bcal$ and $Y\colon \Mcal \to \Bcal$ is proper. For the first assertion, let $P\in \Mcal$ and take $k\in\n$ such that $P\in \Ucal_k$. From {\rm (E$_n$)} and the Convex Hull Property, $X_n(P)\in \Dcal^n$ for all $n\geq k$. Taking limits as $n\to\infty$, we obtain that $Y(P)\in\overline{\Bcal}$ and so, by the convexity of $\Bcal$ and the Maximum Principle for harmonic functions, $Y(P)\in\Bcal$.

Then, properties {\rm (F$_n$)}, $n\in\n$, and the fact that  $\{ \overline{\Dcal^{n-1}_{-\epsilon_n}} \}_{n \in \n}$ is an exhaustion by compact sets of $\Bcal$ imply that 
\begin{equation}\label{eq:CD1}
Y(\Mcal \setminus \overline{\Ucal}_{k-1}) \subset \Bcal \setminus \Dcal^{k-1}_{-\epsilon_k}\quad \text{for all $k\in\n$.}
\end{equation}
This inclusion for $k=1$ proves {\rm (iv)}. To check that $Y\colon \Mcal \to \Bcal$ is proper, let $K\subset \Bcal$ be a compact subset. Since  $\{ \overline{\Dcal^{n-1}_{-\epsilon_n}} \}_{n \in \n}$ is an exhaustion of $\Bcal$, there exists $k\in \n$ such that $K\subset  \Dcal^{n-1}_{-\epsilon_n}$ for all $n\geq k.$ Therefore, \eqref{eq:CD1} gives that $Y^{-1}(K) \subset \overline{\Ucal}_{k-1}.$ This shows that $Y^{-1}(K)$ is compact and proves {\rm (iii)}.

This completes the proof.
\end{proof}

%%%%%%%%%%%%%%%%%
%%%%%%%%%%%%%%%%%
%%%%%%%%%%%%%%%%%

\section{Approximation by embedded complex curves}\label{sec:lemma-proof}

In this section we prove Lemma \ref{lem:compilation}. The proof consists of three main steps. In the first step (Subsec.\ \ref{subsec:net}), we introduce the notion of {\em tangent net} for a convex domain, and prove an existence result of tangent nets with useful geometrical properties. The second step is an approximation result by complex curves along tangent nets; see Subsec.\ \ref{subsec:lemma}. In the final step we prove a desingularization result for complex curves in $\c^2$; see Subsec.\ \ref{subsec:embe}. Lemma \ref{lem:compilation} will follow by combining these results; see Subsec.\ \ref{subsec:proofcompi}.

%%%%%%%%%%%%%%%%%
%%%%%%%%%%%%%%%%%
%%%%%%%%%%%%%%%%%

\subsection{Tangent nets}\label{subsec:net}

The aim of this section is to introduce the notion of {\em tangent net} (Def.\ \ref{def:net}) and prove an existence result of tangent nets with useful properties for our purposes (see Lemma \ref{lem:net}).

%Let $\Dcal$ be a regular strictly convex bounded domain in $\r^n.$ For any point $p\in\Fr\Dcal$, there exists ${\bf u}_p \in (T_p\Fr\Dcal)\cap\s^3$ such that
%\begin{equation}\label{eq:CxR}
%\lP {\bf u}_p, \nu_\Dcal(p) \rP=0 \quad\text{and}\quad T_p\Fr\Dcal = {\rm span}_\c ({\bf u}_p) \oplus {\rm span}_\r(J(\nu_\Dcal(p))),
%\end{equation}
%where $J\colon \r^n\to\r^n$, $J(z_1,z_2)=\imath(z_1,z_2)$, is the canonical complex structure on $\r^n$ (recall that we denote by $\lP\cdot,\cdot\rP$ the usual Hermitian inner product in $\r^n$).

\begin{defin}\label{def:net}
Let $\Dcal$ be a bounded regular strictly convex domain in $\r^n,$ $n\geq 2$. Let $\Delta \subset \Fr \Dcal$ be a finite set 
and call
\[
\Gamma:=\bigcup_{p\in \Delta} (p+T_p\, \Fr\Dcal)\subset \r^n\setminus \Dcal.
\]

The set
\[
\Tcal:=\{q\in \r^n\colon \dist(q,\Gamma)<\epsilon\}%\Subset\r^n
\]
is said to be a {\em tangent net} of radius $\epsilon>0$ for $\Dcal.$ (See Fig.\ \ref{fig:net}.)  Observe that if $\epsilon<1/\kappa(\Dcal)$ then $\Tcal\subset \r^n\setminus \overline{\Dcal}_{-\epsilon}$.

The sets $\Tcal^0:=\Delta$ and $\Tcal^1:=\Gamma$ are said to be the $0$-{\em skeleton}  and the $1$-{\em skeleton} of $\Tcal,$ respectively. 
For any $p\in\Tcal^0,$ the set $\Tcal(p):=\{q\in\r^n\colon \dist(q,p+T_p\Fr\Dcal)<\epsilon\}$ 
is said to be the {\em slab of $\Tcal$ based at $p$}.
\end{defin}

\begin{figure}[ht]
    \begin{center}
    \scalebox{0.35}{\includegraphics{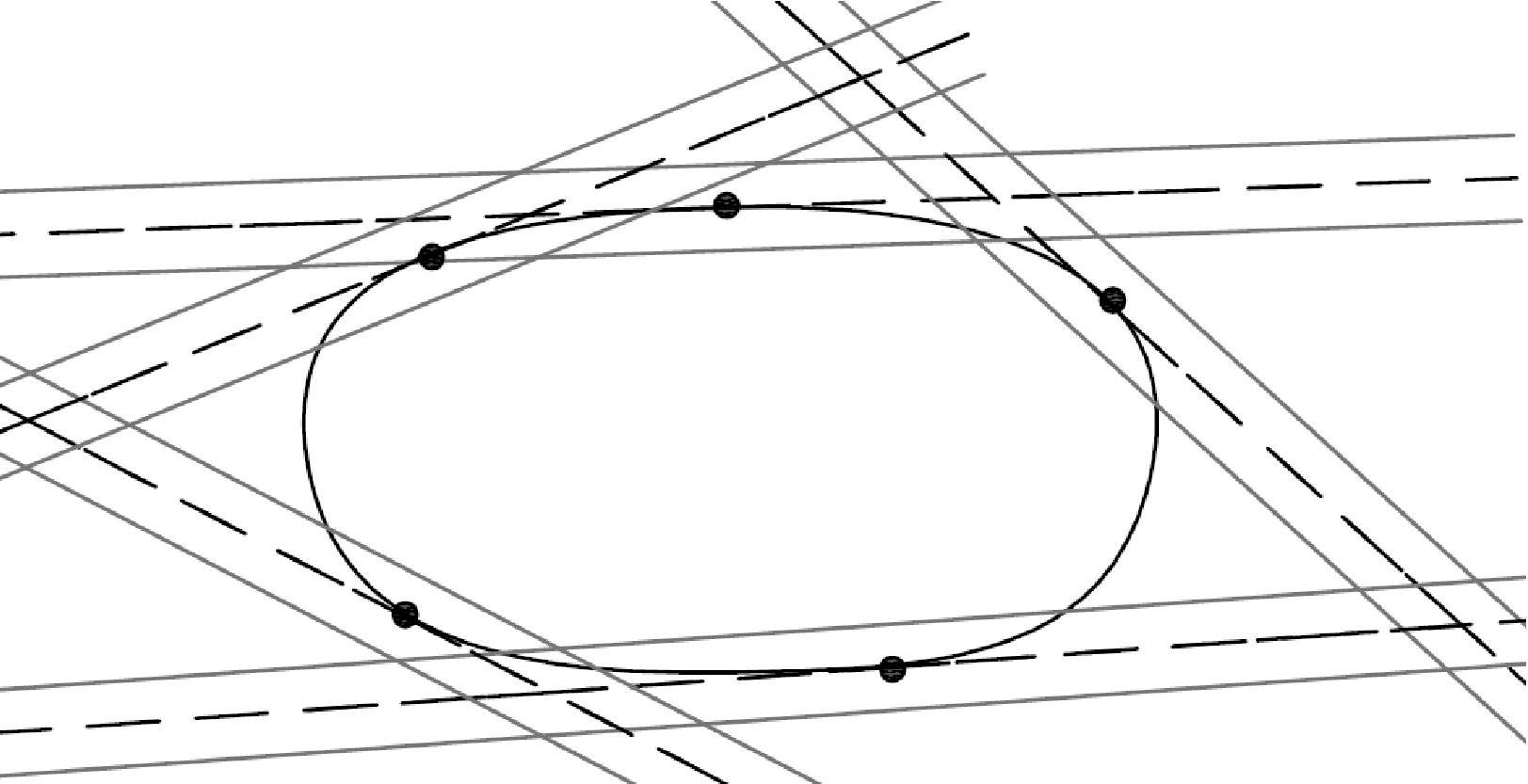}}
        \end{center}
%        \vspace{-0.5cm}
\caption{A tangent net}\label{fig:net}
\end{figure}
 
The following {\em Pythagoras' type} result will be crucial in this paper.  

\begin{lem}\label{lem:net} 
Let $\Dcal$ and $\Dcal'$ be bounded regular strictly convex domains in $\r^n$ ($n\geq 2$), $\Dcal\Subset\Dcal'$. Let $A\subset\Fr\Dcal$ consisting of a finite collection of smooth immersed compact arcs and closed curves.

Then for any $\epsilon>0$ there exists a tangent net $\Tcal$ of radius $<\epsilon$ for $\Dcal$ such that 
\begin{enumerate}[\rm (i)]
\item $A\subset \Tcal$ and
\item $\ell(\gamma)> {\bf d}(\Dcal,\Fr\Dcal')-\epsilon$ for any Jordan arc $\gamma\subset\Tcal$ connecting $\Fr\Dcal$ and $\Fr\Dcal'$.

\end{enumerate}

\end{lem}

\begin{proof}
For the sake of simplicity, denote by $d_0:=\dist(\Dcal,\Fr\Dcal')$ and $\kappa_0:=\kappa(\Dcal)$.

Write $A=\cup_{i=1}^\mu \alpha_i$, where $\alpha_i$ is either a smooth closed immersed curve or a smooth immersed compact arc in $\Fr\Dcal$ for all $i\in\{1,\ldots,\mu\}$, $\mu\in\n$. Denote by 
\begin{equation}\label{eq:L}
\Lgot:=1+\max\{\ell(\alpha_i)\colon i=1,\ldots,\mu\}<+\infty.
\end{equation}

For any $m\in\n$ we set 
\begin{equation}\label{eq:epsilon_m}
\epsilon_m:= \frac{1}{\kappa_0} \Big(1-\cos\Big(\frac{\Lgot \kappa_0}{m}\Big)\Big).
\end{equation}
Since $\lim_{m\to\infty} m\epsilon_m=0$, then
\begin{equation}\label{eq:epsilon_m2}
\max\left\{\epsilon_m \,,\,\frac{4 (m \mu+1) \epsilon_m}{\sqrt{(d_0\kappa_0+1)^2-1}} \right\}<\epsilon
\end{equation}
for large enough $m.$ 

Let $m\in \n$ satisfying \eqref{eq:epsilon_m2} and call $I:=\{1,\ldots,\mu\}\times \{1,\ldots,m\}$.

From \eqref{eq:L}, for any $i\in\{1,\ldots,\mu\}$ there exist $m$ points $p_{i,1},\ldots p_{i,m}$ splitting $\alpha_i$ into $m$ arcs of the same length $<\Lgot/m$. Denote by
$\Delta:=\{p_{i,j}\colon (i,j)\in I\}$, let $\Tcal$ be the tangent net of radius $\epsilon_m$ for $\Dcal$ with $0$-skeleton $\Tcal^0=\Delta$, and observe that 
\begin{equation} \label{eq:t}
\text{$\dist_{\Fr \Dcal}(q,\Tcal^0)<\Lgot/m$ for all $q\in A$,}
\end{equation}
where $\dist_{\Fr \Dcal}$ is the intrinsic distance in $\Fr \Dcal$.

Let us show that $\Tcal$ solves the lemma. 
 
First, let us check  item {\rm (i)}. In view of \eqref{eq:t}, it suffices to check that the slab $\Tcal(p_{i,j})$
contains the intrinsic geodesic ball in $\Fr\Dcal$ with center $p_{i,j}$ and radius $\Lgot/m$, for all $(i,j)\in I$. Indeed, let $S_{i,j}\subset\overline{\Dcal}$ denote the Euclidean sphere in $\r^n$ of radius $1/\kappa_0$ tangent to $\Fr\Dcal$ at $p_{i,j}$. Basic trigonometry and \eqref{eq:epsilon_m} give that $\Tcal(p_{i,j})$ contains the intrinsic geodesic ball in $S_{i,j}$ with center $p_{i,j}$ and radius $\Lgot/m$. Then the assertion follows from Rauch's theorem and the definition of $\kappa_0$ (see \eqref{eq:kD}).

Let us show that $\Tcal$ satisfies item {\rm (ii)}. Let $\gamma\subset\Tcal$ be as in {\rm (ii)} and denote by $p_0\in \Fr\Dcal$ and $q_0\in \Fr\Dcal'$ the endpoints of $\gamma$. Without loss of generality, assume that $\gamma\subset\Tcal\cap\overline{\Dcal'}$. Let $C$ be the cone in $\r^n$ given by
\[
C:=\bigcup_{x\in \Lambda} [x,q_0],\quad \text{where $\Lambda:={\{x\in \Fr\Dcal \colon q_0\in x+T_x\Fr\Dcal\}}.$}
\]
Denote by $\Omega$ the compact region in  $\r^n\setminus \Dcal$ bounded by $\Fr\Dcal$ and $C$; see Figure \ref{fig:omega0}. 
\begin{figure}[ht]
    \begin{center}
    \scalebox{0.26}{\includegraphics{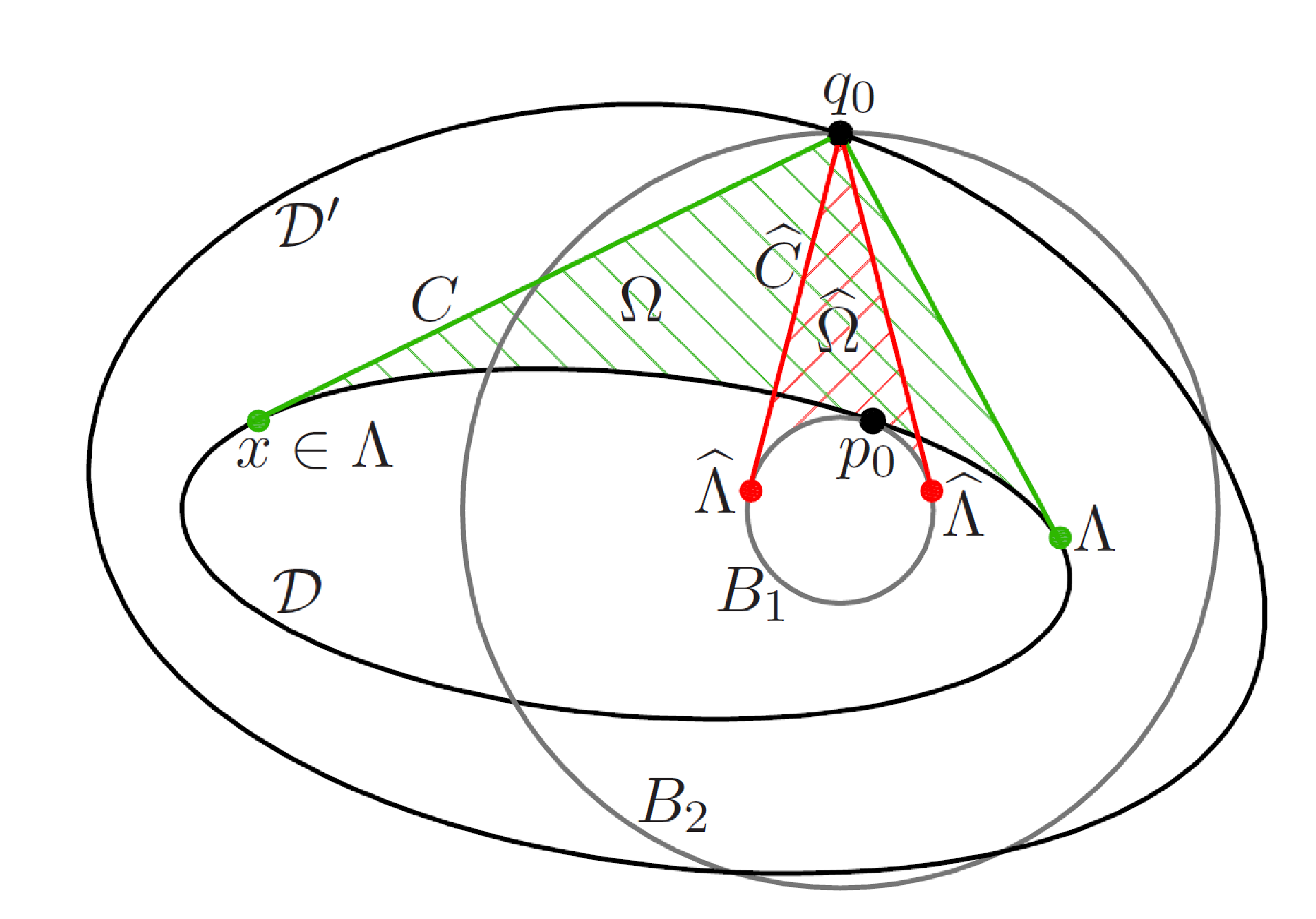}}
        \end{center}
%        \vspace{-0.5cm}
\caption{Proof of Lemma \ref{lem:net}}\label{fig:omega0}
\end{figure}

Assume first that $p_0\in \Fr \Dcal \setminus \Fr\Omega$. In this case there exists $x_0\in \Lambda$ such that $\ell(\gamma) \geq \ell([x_0,q_0])$. Since $\Dcal$ and $\Dcal'$ are strictly convex, then the definition of $\kappa_0$ and Pithagoras' theorem give that 
\[
\ell([x_0,q_0]) \geq  \sqrt{d_0^2+2\frac{d_0}{\kappa_0}} > {\bf d}(\Dcal,\Fr\Dcal'),
\]
and we are done; the latter inequality follows from a straightforward computation. 

Assume now that $p_0\in \Fr\Dcal \cap \Fr \Omega$. 
Let $B_1\subset\Dcal$ be the Euclidean open ball in $\r^n$ of radius $1/\kappa_0$ tangent to $\Fr\Dcal$ at $p_0$. Let $B_2$ be the Euclidean open ball in $\r^n$ with the same center as $B_1$ and such that $q_0\in \Fr B_2$. Denote by
\begin{equation} \label{eq:lambda}
\wh \Lambda :=\{x\in \Fr B_1 \colon q_0\in x+T_x \Fr B_1\},\enskip \wh C:=\bigcup_{x\in \wh \Lambda} [x,q_0],\enskip\text{and}\enskip \wh \Tcal:=\Tcal \cup\wh C.
\end{equation} 
Denote by $\wh \Omega$ the compact region in  $\r^n\setminus B_1$ bounded by $\Fr B_1$ and $\wh C$, and notice that $\wh\Omega\subset\overline{\Dcal'}$; see Figure \ref{fig:omega0}. 
Since $[p_0,q_0]\cap \overline{\Dcal}=\{p_0\}$ and $p_0\in \overline{B}_1 \subset \overline{\Dcal},$ then  $[p_0,q_0]\cap \overline{B}_1=\{p_0\}$ as well, and so $p_0 \in \Fr \Dcal \cap \Fr \wh \Omega$.

If $\gamma\cap (\r^n\setminus\wh\Omega)\neq\emptyset$, let $p_1$ be the first  point of $\gamma$  in $\wh C$ and let $\gamma_0\subset \gamma\cap \wh \Omega$ be the sub-arc of $\gamma$ with endpoints $p_0$ and $p_1.$ Observe that the arc $\wh \gamma_0:= \gamma_0 \cup [p_1,q_0]\subset \wh \Tcal\cap\wh\Omega$ connects $p_0$ and $q_0$ and satisfies $\ell(\wh \gamma_0)\leq\ell(\gamma)$. 
Therefore, to finish the proof it suffices to show that $\ell(\wh\gamma) >{\bf d}(\Dcal,\Fr \Dcal')-\epsilon$ for any compact arc $\wh\gamma\subset\wh\Tcal\cap\wh\Omega$ with endpoints $p_0$ and $q_0$. Let $\wh \gamma$ be such an arc.

Up to a rigid motion, assume that $B_1$ and $B_2$ are centered at $\vec 0\in\r^n$ and $q_0=(\vec 0,r_2)\in\r^{n-1}\times \r$, where $r_2$ is the radius of $B_2$. Since the radius of $B_1$ equals $1/\kappa_0$, $p_0\in \Fr B_1\cap \Fr \Dcal$, and $q_0\in \Fr B_2\cap \Fr \Dcal'$, it follows that
\begin{equation}\label{eq:r_2}
r_2\geq d_0+1/\kappa_0.
\end{equation}
In this setting, the set $\wh\Lambda$ in \eqref{eq:lambda} is
\begin{equation}\label{eq:altura}
\wh \Lambda=\left\{\left(\vec x\,,\,\frac{1}{r_2\kappa_0^2}\right) \in \r^{n-1}\times\r\colon \|\vec x\|=\frac{\sqrt{r_2^2\kappa_0^2-1}}{r_2 \kappa_0^2}\right\}.
\end{equation}

Since the endpoint $q_0$ of $\wh \gamma$ is the vertex of the cone $\wh C$ (see \eqref{eq:lambda}), then there exist $a\in\n$ satisfying 
\begin{equation} \label{eq:a-1}
a-1\leq \sharp I=m\mu,
\end{equation}
a compact polygonal arc $\beta=\cup_{i=1}^a L_i \subset \wh\Tcal\cap\wh\Omega$
 with endpoints $p_0$ and $q_0$, and an injective map $\{1,\ldots,a-1\}\ni i\mapsto \sigma_i\in I$, such that:
\begin{itemize}
\item $L_i=[(\vec x_i,y_i),(\vec x_{i+1},y_{i+1})]\subset \r^{n-1}\times \r,$ $i=1,\ldots,a$.
\item $(\vec x_1,y_1)=p_0$ and $(\vec x_{a+1},y_{a+1})=q_0=(\vec 0,r_2)$ in $\r^{n-1}\times\r$.
\item $L_i\subset \Tcal_{\sigma_i}$ for all $i=1,\ldots,a-1$.
\item $L_a\subset\wh C$ (possibly $L_a=\{q_0\}$).
\item $\ell(\beta)\leq\ell(\wh\gamma)$.
\end{itemize}

To finish it suffices to show that $\ell(\beta)> {\bf d}(\Dcal,\Fr\Dcal)-\epsilon$. 

Since $\Tcal$ is a tangent net of radius $\epsilon_m$ for $\Dcal$ and the slope of any segment in $\Tcal^1\cap\wh\Omega$ is at most the one of the cone $\wh C$ (that is to say, the slope of the segment  $[q_1,q_0]$ over $\r^{n-1}\times\{0\}$ for any $q_1=(\vec x_{q_1},y_{q_1})\in \wh \Lambda$, which equals $(r_2-y_{q_1})/\|\vec{x}_{q_1}\|=\sqrt{r_2^2\kappa_0^2-1}$), then basic trigonometry gives that
\begin{equation}\label{eq:slope}
h_i:=|y_{i+1}-y_i|\leq \|\vec x_{i+1}-\vec x_i\|\sqrt{r_2^2\kappa_0^2-1}+2r_2\kappa_0\epsilon_m \quad \forall i\in\{1,\ldots,a\}.
\end{equation}

Since $(\vec x_1,y_1)=p_0 \in \Fr B_1=\s^{n-1}(1/\kappa_0)$ then $y_1\leq 1/\kappa_0$; and since  $y_{a+1}=r_2$, then
\begin{equation}\label{eq:y}
\sum_{i=1}^a h_i \geq r_2-1/\kappa_0.
\end{equation}

From \eqref{eq:slope}, one obtains that
\begin{equation}\label{eq:F-G}
\ell(\beta) = \sum_{i=1}^a \ell(L_i) = \sum_{i=1}^a \sqrt{\|\vec x_{i+1}-\vec x_i\|^2+h_i^2} \geq  F-G,
\end{equation}
where
\[
F=\frac{r_2 \kappa_0}{\sqrt{r_2^2 \kappa_0^2-1}}\sum_{i=1}^a h_i
\]
and
\[
G=\frac{r_2 \kappa_0}{\sqrt{r_2^2 \kappa_0^2-1}}\sum_{i=1}^a \Big(h_i-\frac{1}{r_2\kappa_0}\sqrt{(h_i-2\epsilon_mr_2\kappa_0)^2+(r_2^2\kappa_0^2-1)h_i^2}\Big).
\]

On the one hand, since the function 
\[
f\colon ]\frac{1}{\kappa_0},+\infty[\,\to\, ]0,+\infty[,\quad 
f(t)=\frac{t^2 \kappa_0-t}{\sqrt{t^2 \kappa_0^2-1}},
\]
is increasing, one infers from \eqref{eq:y} and  \eqref{eq:r_2} that 
\begin{equation}\label{eq:F}
 F\geq \frac{r_2^2 \kappa_0-r_2}{\sqrt{r_2^2 \kappa_0^2-1}} = f(r_2) \geq  f(d_0+1/\kappa_0) =  {\bf d}(\Dcal,\Fr \Dcal').
\end{equation}
On the other hand, one has
\begin{eqnarray*}
G & = & \frac{r_2 \kappa_0}{\sqrt{r_2^2 \kappa_0^2-1}}\sum_{i=1}^a \frac{-4\epsilon_m^2+\frac{4\epsilon_m}{r_2\kappa_0} h_i}{h_i+\frac{1}{r_2\kappa_0}\sqrt{(h_i-2\epsilon_mr_2\kappa_0)^2+(r_2^2\kappa_0^2-1)h_i^2}}\\
& < & \frac{4\epsilon_m}{\sqrt{r_2^2 \kappa_0^2-1}}  \sum_{i=1}^a \frac{h_i}{h_i+\frac{1}{r_2\kappa_0}\sqrt{(h_i-2\epsilon_mr_2\kappa_0)^2+(r_2^2\kappa_0^2-1)h_i^2}}\\
&\leq & \frac{4\epsilon_m a}{\sqrt{r_2^2 \kappa_0^2-1}}.
\end{eqnarray*}
Therefore, taking into account  \eqref{eq:r_2}, \eqref{eq:a-1}, and \eqref{eq:epsilon_m2}, one gets
\[
G < \frac{4 (m\mu+1)\epsilon_m}{\sqrt{(d_0\kappa_0+1)^2-1}} <\epsilon.
\]
This inequality, \eqref{eq:F-G}, and \eqref{eq:F} prove the lemma.
\end{proof}

%%%%%%%%%%%%%%%%%
%%%%%%%%%%%%%%%%%
%%%%%%%%%%%%%%%%%

\subsection{Deforming curves along tangent nets}\label{subsec:lemma}

The following approximation result is the second key in the proof of Lemma \ref{lem:compilation}. See Def.\ \ref{def:net} for notation.

\begin{lem}\label{lem:main}
Let $\Dcal$ and $\Dcal'$ be bounded regular strictly convex domains in $\c^2$, $\Dcal\Subset\Dcal'$. Let $\varepsilon\in ]0,\min \{\dist(\overline{\Dcal},\Fr\Dcal'),1/\kappa(\Dcal)\}[$ and let $\Tcal$ be a tangent net of radius $\varepsilon$ for $\Dcal$.

Let $\delta\in]0,\varepsilon[$, let $\Ncal$ be an open connected Riemann surface equipped with a nowhere-vanishing holomorphic $1$-form $\vartheta_{\Ncal}$, let $\Rcal\Subset\Ncal$ be a bordered domain, and let $X\colon \overline{\Rcal}\to \c^2$ be a holomorphic immersion such that %$X(\Rcal)\subset \Dcal$ and
\begin{equation}\label{eq:lemma}
X(b\overline{\Rcal})\subset \Tcal\cap \Dcal_{\delta} \quad \text{(hence $X(\overline{\Rcal})\subset \Dcal_{\delta}$).}
\end{equation}

Then there exist a bordered domain $\Scal\Subset\Ncal$ and a holomorphic immersion $Y\colon \overline{\Scal}\to\c^2$ enjoying the following properties:
\begin{enumerate}[\rm (a)]
\item $\Rcal\Subset \Scal$ and $\Rcal$ and $\Scal$ are homeomorphically isotopic (i.e., $\overline{\Scal}\setminus\Rcal$ consists of a finite collection of pairwise disjoint compact annuli).
\item $\|Y-X\|_{1,\overline{\Rcal};\vartheta_\Ncal}<\delta$ (see \eqref{eq:norma1}).
\item $Y(\overline{\Scal}\setminus {\Rcal})\subset \overline{\Dcal'}\setminus \overline{\Dcal}_{-\varepsilon}$.
\item $Y(b\overline{\Scal})\subset \Fr\Dcal',$ hence $Y(\overline{\Scal})\subset \overline{\Dcal'}$.
\item $Y(\overline{\Scal})  \subset \Dcal_\delta \cup \Tcal$.
\end{enumerate}
\end{lem}
Before going into the proof of Lemma \ref{lem:main}, let us say a word about its geometrical implications. 
Roughly speaking, the lemma ensures that an immersed compact complex curve $X(\overline{\Rcal})\subset\c^2$ with boundary $X(b\overline{\Rcal})$ lying close to the frontier of a regular strictly convex domain $\Dcal\Subset\c^2$, can be approximated by another one $Y(\overline{\Scal})\subset\c^2$ with boundary $Y(b\overline{\Scal})$ in the frontier of a larger convex domain $\Dcal'$. The main point is that this can be done in such a way that the piece of $Y(\overline{\Scal})$ outside $\Dcal$ lies in a given tangent net $\Tcal$ for $\Dcal$ containing $X(b \overline{\Rcal})$; see Lemma \ref{lem:main}-{\rm (e)}.

Notice that the intrinsic Euclidean diameter of the complex curve $Y \colon \Scal \to \c^2$ exceeds in $\dist(\Dcal,\Fr\Dcal')$ the one of $X\colon\Rcal\to\c^2$. Combining this lemma with a suitable choice of $\Tcal$ accordingly to Lemma \ref{lem:net}, one can also guarantee that the {\em image diameter} of the curve $Y$ exceeds in ${\bf d}(\Dcal,\Fr\Dcal')$ the one of the initial curve $X$ (see Def.\ \ref{def:d} and \ref{def:extrinsic}). 
This fact will be the key for obtaining image completeness  while preserving boundedness in the proof of Theorem \ref{th:intro2} (Sec.\ \ref{sec:extrinsic}). The main novelty of Lemma \ref{lem:main} with respect to previous related constructions (cf. \cite{Nadirashvili,AL-CY,AF-1} and references therein) is to estimate the image diameter of the curve instead of the intrinsic one.

From the technical point of view, the proof of the lemma relies on approximating $X(\overline{\Rcal})$ by another immersed curve $\overline{\Sigma}\subset\c^2$ with boundary $b\overline{\Sigma}$ in $\c^2\setminus\overline{\Dcal'}$, such that $\overline{\Sigma}\subset \Dcal_\delta \cup \Tcal$. Lemma \ref{lem:main} will follow up to trimming off the curve $\overline{\Sigma}$ in order to ensure item {\rm (d)}. 
The construction of the immersed compact complex curve $\overline{\Sigma}$ depends  on the classical Runge and Mergelyan approximation theorems, and consists of three main steps that we now roughly describe. 

First, we split the boundary $b\overline{\Rcal}$ into a finite collection of pairwise disjoint Jordan arcs $\alpha_{i,j}$ so that $X(\alpha_{i,j})$ lies in a slab $\Tcal(p_{i,j})$ of $\Tcal,$ $p_{i,j}\in \Tcal^0$; see items {\rm (i)}-{\rm (iv)} below.  

In the second step (properties {\rm (v)}-{\rm (vii)} below), we attach to $X(\overline{\Rcal})$ a family of Jordan arcs $\lambda_{i,j}\subset\c^2$ with initial point at an endpoint of $X(\alpha_{i,j})\subset X(b\overline{\Rcal})$ and final point in $\c^2\setminus\overline{\Dcal'}$.  Each $\lambda_{i,j}$ is chosen to be close to a segment inside the slab $\Tcal(p_{i,j})$. We then approximate $X(\overline{\Rcal})\cup(\cup_{i,j}\lambda_{i,j})$ by a new curve $F(\overline{\Mcal})$, $\Rcal\Subset\Mcal\Subset\Ncal$ (see properties {\rm (viii)}-{\rm (xiii)} below). The bordered domain $\Mcal$ is chosen so that the final point of $r_{i,j}=F^{-1}(\lambda_{i,j})$ lies in $b \overline{\Mcal}.$

In the final step, we first split the boundary $b\overline{\Mcal}$ into finitely many arcs $\beta_{i,j}$ coordinately to the $\alpha_{i,j}$'s and the $r_{i,j}$'s (properties {\rm (xiv)}-{\rm (xvi)} below). 
The arcs $r_{i,j}$'s split $\Mcal\setminus \overline{\Rcal}$ into a finite collection of topological discs $\Agot_{i,j},$ where $\alpha_{i,j} \cup \beta_{i,j} \subset \Fr \overline{\Agot}_{i,j}$. Then, we {\em stretch}  $F(\Agot_{i,j})$  outside of $\overline{\Dcal'}$ along the slab $\Tcal(p_{i,j})$ in a complex direction {\em orthogonal} to $\lambda_{i,j},$ hence preserving the already done in the second step. This gives a curve $\overline{\Sigma}$ as the one announced above ($\overline{\Sigma}$ corresponds to $Y_{n}(\overline{\Mcal})$ for $n={\tt I}{\tt J},$ see properties {\rm (1$_n$)}-{\rm (6$_n$)} below).

\begin{proof}[Proof of Lemma \ref{lem:main}]
Recall that $\lP\cdot,\cdot\rP$ denotes the bilinear Hermitian product of $\c^2$ and $\nu_\Dcal\colon \Fr \Dcal \rightarrow \s^3$ the
outward pointing unit normal of $\Fr \Dcal$. Denote by $\Jscr\colon\c^2\to\c^2$, $\Jscr(\zeta,\xi)=(\imath\zeta,\imath\xi)$, the canonical complex structure of $\c^2$.

We begin with the following reduction. Since $\nu_\Dcal\colon \Fr \Dcal \rightarrow \s^3$ is a diffeomorphism,  we can assume without loss of generality that 
\begin{equation}\label{eq:net0}
\lP \nu_\Dcal(p_1) \rP^\bot \cap \lP \nu_\Dcal(p_2) \rP^\bot=\{0\} \quad \forall \{p_1, p_2\}\subset \Tcal^0,\, p_1\neq p_2.
\end{equation}
Indeed,  just replace $\Tcal$ by another tangent net $\wh \Tcal$ for $\Dcal$ satisfying   $X(b\overline{\Rcal})\subset \wh \Tcal$, $\wh \Tcal\cap\overline{\Dcal'}\subset\Tcal\cap\overline{\Dcal'}$, and \eqref{eq:net0}. To do so, choose $\wh\Tcal$ with $0$-skeleton and radius ($<\varepsilon$) close enough to the ones of $\Tcal$ and use the fact that  condition \eqref{eq:net0} determines and open and dense subset in the space of tangent nets for $\Dcal$.

Since $\lP \nu_\Dcal(p) \rP^\bot= T_{p}\Fr\Dcal \cap \Jscr(T_{p}\Fr\Dcal)$ for all $p \in \Fr \Dcal$, equation \eqref{eq:net0} yields that $(T_{p_1}\Fr\Dcal \cap T_{p_2}\Fr\Dcal) \setminus (\lP \nu_\Dcal(p_1) \rP^\bot\cup \lP \nu_\Dcal(p_2) \rP^\bot)\neq\emptyset$  for any couple $\{p_1, p_2\}\subset \Tcal^0$, $p_1\neq p_2$. For every couple $\{p_1, p_2\}\subset \Tcal^0$, $p_1\neq p_2$, fix  
\begin{equation}\label{eq:vp1p2}
{\sf v}_{\{p_1,p_2\}}\in \big( (T_{p_1}\Fr\Dcal \cap T_{p_2}\Fr\Dcal) \setminus (\lP \nu_\Dcal(p_1) \rP^\bot\cup \lP \nu_\Dcal(p_2) \rP^\bot)\big) \cap\s^3. 
\end{equation}

The {\em first step} of the proof consists of suitably splitting the boundary curves of $\overline{\Rcal}$.  
Denote by $\alpha_1,\ldots,\alpha_{\tt I}$, ${\tt I}\in\n$, the connected components of $b\overline{\Rcal}$, which are smooth Jordan curves in $\Ncal$. From \eqref{eq:lemma}, there exist a natural number ${\tt J} \geq 3,$  a family of Jordan sub-arcs $\{\alpha_{i,j}\subset\alpha_i \colon (i,j)\in \Hgot:=\{1,\ldots,{\tt I}\}\times \z_{\tt J}\}$ (here $\z_{\tt J}=\{0,1,\ldots,{\tt J}-1\}$  denotes the additive cyclic group of integers modulus ${\tt J}$), and points $\{p_{i,j} \colon (i,j)\in \Hgot\}\subset\Tcal^0$, meeting the following requirements:
\newcounter{saveenum}
\begin{enumerate}[\rm (i)]
\item $\cup_{j=1}^{{\tt J}} \alpha_{i,j}= \alpha_i,$ for all $i\in\{1,\ldots,{\tt I}\}$.
\item $\alpha_{i,j}\cap\alpha_{i,k}=\emptyset$ for all $(i,j)\in \Hgot$ and $k\in\z_{\tt J}\setminus\{j-1,j,j+1\}$.
\item $\alpha_{i,j}$ and $\alpha_{i,j+1}$ have a common endpoint $Q_{i,j}$ and are otherwise disjoint for all  $(i,j)\in \Hgot$. 
\item $X(\alpha_{i,j})  \subset \Tcal(p_{i,j})\cap \Dcal_{\delta}$ for all  $(i,j)\in \Hgot$, where 
$\Tcal(p_{i,j})$ is the slab of  $\Tcal$ based at $p_{i,j}\in\Tcal^0$ (see Def.\ \ref{def:net}). 
\setcounter{saveenum}{\value{enumi}}
\end{enumerate}
To find such a partition, choose the arcs $\alpha_{i,j}$ so that $X(\alpha_{i,j})\subset\c^2$  has sufficiently small diameter  for all $(i,j)\in \Hgot$. Take into account \eqref{eq:lemma} in order to ensure {\rm (iv)}. Notice that the map $\Hgot \ni (i,j) \mapsto p_{i,j}\in \Tcal^0$ is not necessarily either injective or surjective.

In the {\em second step} we attach to $X(\overline{\Rcal})$ a suitable family of Jordan arcs. In the Riemann surface $\Ncal,$ take  for every $(i,j)\in \Hgot$ an analytic Jordan arc $r_{i,j}\subset \Ncal\setminus \Rcal$  attached transversally to $b\overline{\Rcal}$ at $Q_{i,j}$ and otherwise disjoint from $\overline{\Rcal}$. In addition, choose these arcs to be pairwise disjoint. Denote by $P_{i,j}$ the other endpoint of $r_{i,j}$, $(i,j)\in \Hgot$. 

For every $(i,j)\in \Hgot$, there exists a smooth regular embedded arc $\lambda_{i,j}$ in $\c^2$ enjoying the following properties:
\begin{enumerate}[\rm (i)]
\setcounter{enumi}{\value{saveenum}}
\item $\lambda_{i,j}\subset \Tcal(p_{i,j})\cap \Tcal(p_{i,j+1})$. In particular,  $\lambda_{i,j}+T_{p_{i,k}}\Fr\Dcal:=\cup_{q\in \lambda_{i,j}} (q+T_{p_{i,k}}\Fr\Dcal) \subset \Tcal(p_{i,k}) \subset\Tcal$ for $k=j,j+1$.
\item $\lambda_{i,j}$ is attached transversally to $X(b\overline{\Rcal})$ at $X(Q_{i,j})$ and  matches smoothly with $X(\overline{\Rcal})$ at $X(Q_{i,j})$.
\item  $|\langle o_{i,j}-X(Q_{i,j}),\Jscr(\nu_\Dcal(p_{i,k})) \rangle| >1+{\rm diam}(\Dcal')$, for $k=j,j+1$, where $o_{i,j}$ is the endpoint of $\lambda_{i,j}$,  $o_{i,j}\neq X(Q_{i,j})$ (recall that $\langle \cdot,\cdot\rangle$ denotes the Euclidean inner product).
\setcounter{saveenum}{\value{enumi}}
\end{enumerate}
Indeed, the arc $\lambda_{i,j}$ can be obtained as a slight deformation of the segment 
\[
[X(Q_{i,j}),X(Q_{i,j})+c_{i,j} {\sf v}_{\{p_{i,j},p_{i,j+1}\}}]\subset\c^2,
\]
where ${\sf v}_{\{p_{i,j},p_{i,j+1}\}}$  is given by \eqref{eq:vp1p2}
and $c_{i,j}>0$ is a large enough constant so that the above segment formally meets {\rm (vii)} (notice that $\langle {\sf v}_{\{p_{i,j},p_{i,j+1}\}},\Jscr(\nu_\Dcal(p_{i,k}))\rangle\neq 0,$ $k=j,j+1$; see \eqref{eq:vp1p2}). For item {\rm (v)}, take into account {\rm (iii)}, {\rm (iv)}, and \eqref{eq:vp1p2}. Further, {\rm (vi)} trivially follows up to a slight deformation of the segment.

Extend $X$, with the same name, to a smooth function $\overline{\Rcal}\cup(\cup_{(i,j)\in \Hgot} r_{i,j})\to\c^2$ mapping the arc $r_{i,j}$ diffeomorphically onto $\lambda_{i,j}$ for all $(i,j)\in \Hgot$.
In this setting, Mergelyan's theorem furnishes a bordered domain $\Mcal\Subset\Ncal$ and a holomorphic immersion 
\[
Y_0\colon \overline{\Mcal}\to\c^2,
\]
as close as desired to $X$ in the $\Ccal^1$ topology on $\overline{\Rcal}\cup(\cup_{(i,j)\in \Hgot} r_{i,j})$,   such that:
\begin{enumerate}[\rm (i)]
\setcounter{enumi}{\value{saveenum}}
\item $\Rcal\Subset\Mcal$ and $\overline{\Mcal}\setminus \Rcal$ consists of ${\tt I}$ pairwise disjoint compact annuli $\Agot_1,\ldots,\Agot_{\tt I}$.
\item $\alpha_i\subset\Fr\Agot_i$, $r_{i,j}\subset \Agot_i$, and $r_{i,j}\cap \Fr \Agot_i=\{Q_{i,j},P_{i,j}\}$ for all $(i,j)\in \Hgot$.
%\item $\|{Y_0}-X\|_{0,\overline{\Rcal}\cup(\cup_{(i,j)\in \Hgot} r_{i,j})}<\varepsilon$.
\item $\|{Y_0}-X\|_{1,\overline{\Rcal};\vartheta_\Ncal}<\delta/(1+{\tt I}{\tt J})$, where $\delta>0$ is given in the statement of the lemma.
\item ${Y_0}(r_{i,j})\subset \Tcal(p_{i,j})\cap \Tcal(p_{i,j+1})$ for all $(i,j)\in \Hgot$. See {\rm (v)}.
\item ${Y_0}(\alpha_{i,j})  \subset \Tcal(p_{i,j})\cap \Dcal_\delta$ for all  $(i,j)\in \Hgot$. Take into account {\rm (iv)}.
\item $|\langle {Y_0}(P_{i,j})-{Y_0}(Q_{i,j}),\Jscr(\nu_\Dcal(p_{i,k})) \rangle| >1+{\rm diam}(\Dcal')$, for all $(i,j)\in \Hgot$ and $k\in\{j,j+1\}$. See {\rm (vii)}.
\setcounter{saveenum}{\value{enumi}}
\end{enumerate}

Write $\beta_i=(\Fr \Agot_i)\setminus\alpha_i$ for the  connected component of $\Fr\Agot_i$ disjoint from $\alpha_i$, $i=1,\ldots,{\tt I}$. For every $(i,j)\in \Hgot$ denote by $\Agot_{i,j}$ the connected component of $\Agot_i\setminus \big( \alpha_i \cup (\cup_{j\in\z_{\tt J}} r_{i,j})\big)$ containing $\alpha_{i,j}$ in its frontier. Observe that $\overline{\Agot}_{i,j}$ is a closed disc in $\Agot_i$ bounded by $r_{i,j-1}$, $\alpha_{i,j}$, $r_{i,j}$, and a sub-arc $\beta_{i,j}$ of $\beta_i$ connecting the points $P_{i,j-1}$ and $P_{i,j}$. See Fig.\ \ref{fig:anillo}.
\begin{figure}[ht]
    \begin{center}
    \scalebox{0.35}{\includegraphics{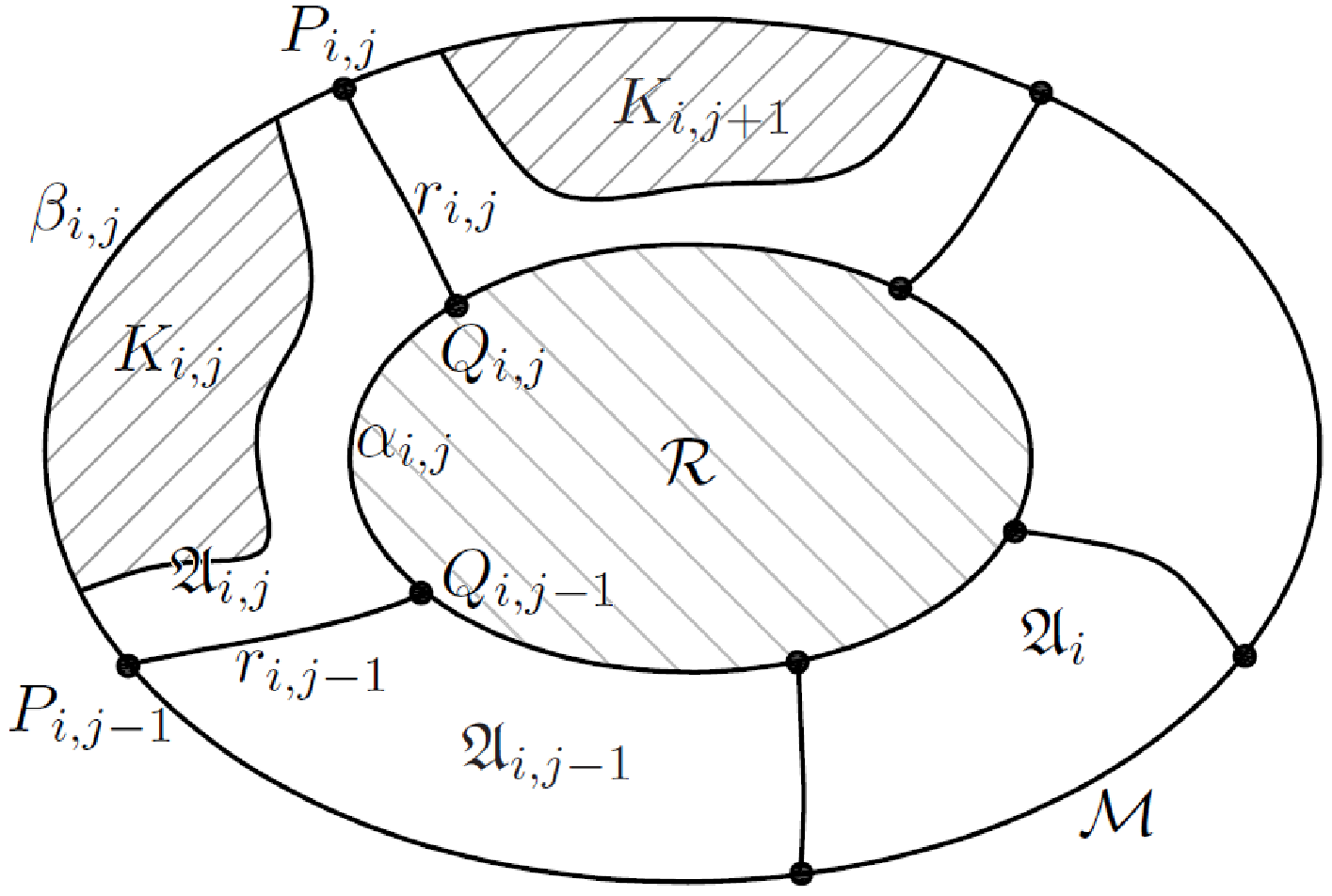}}
        \end{center}
%        \vspace{-0.5cm}
\caption{$\Agot_i$}\label{fig:anillo}
\end{figure}

In the {\em final step} of the construction, we {\em stretch}  $F(\Agot_{i,j})$  outside of $\overline{\Dcal'}$ along the slab $\Tcal(p_{i,j})$.
For every $(i,j)\in \Hgot$, choose a closed disc $K_{i,j}\subset \Agot_{i,j}$ with $\Fr K_{i,j}$ close enough to  $\Fr {\Agot}_{i,j}$ so that:
\begin{enumerate}[\rm (i)]
\setcounter{enumi}{\value{saveenum}}
\item $K_{i,j}\cap\beta_{i,j}$ is a Jordan arc containing neither $P_{i,j-1}$ nor $P_{i,j}$.
\item ${Y_0}(\overline{\Agot_{i,j}\setminus K_{i,j}})\subset \Tcal(p_{i,j})$. Use {\rm (xi)}, {\rm (xii)}, and a continuity argument.
\item $\pi_{i,j}({Y_0}(\overline{\beta_{i,j}\setminus K_{i,j}}))\cap \pi_{i,j}(\overline{\Dcal'})=\emptyset$, where 
\[
\pi_{i,j}\colon \c^2\to {\rm span}_\r(\Jscr(\nu_\Dcal(p_{i,j})))
\]
denotes the orthogonal projection. Use that $\{{Y_0}(Q_{i,j-1}),{Y_0}(Q_{i,j})\}\subset \Dcal'$ (see {\rm (xii)}), property {\rm (xiii)}, and a continuity argument again. See Fig.\ \ref{fig:anillo}.
\setcounter{saveenum}{\value{enumi}}
\end{enumerate}

Let $\sigma\colon \{1,\ldots,{\tt I}{\tt J}\}\to \Hgot$ be a bijective map. To finish, we construct in a recursive process a sequence of holomorphic immersions $Y_n\colon \overline{\Mcal}\to\c^2$, $n\in\{0,1,\ldots,{\tt I}{\tt J}\}$, enjoying the following properties:
\begin{enumerate}[\rm (1{$_n$})]
\item $\|Y_n-Y_{n-1}\|_{1,\overline{\Mcal\setminus \Agot_{\sigma(n)}};\vartheta_\Ncal}<\delta/(1+{\tt I}{\tt J})$.
\item $\lP Y_n-Y_{n-1}, \nu_\Dcal(p_{\sigma(n)}) \rP=0$.
\item $Y_n(\overline{\Agot_{\sigma(a)}\setminus K_{\sigma(a)}})\subset \Tcal_{\sigma(a)}$ for all $a\in \{1,\ldots,{\tt I}{\tt J}\}$.
\item $\pi_{\sigma(a)}(Y_n(\overline{\beta_{\sigma(a)}\setminus K_{\sigma(a)}}))\cap \pi_{\sigma(a)}(\overline{\Dcal'})=\emptyset$ for all $a\in \{1,\ldots,{\tt I}{\tt J}\}$.
%\item $Y_n((\Fr K_{\sigma(a)})\setminus\beta_{\sigma(a)})\subset \Tcal_{\sigma(a)}\setminus\overline{\Dcal'}$.
\item $Y_n(K_{\sigma(a)})\cap \overline{\Dcal'}=\emptyset$ for all $a\in\{1,\ldots,n\}$.
\item $Y_n(\overline{\Rcal})\subset \Dcal_\delta$.
\end{enumerate}

The basis of the induction corresponds to the already given immersion $Y_0.$ Indeed, notice that {\rm (6$_0$)} is implied by {\rm (xii)} and the Convex Hull Property; {\rm (3$_0$)} and {\rm (4$_0$)} agree with {\rm (xv)} and {\rm (xvi)}; and {\rm (1$_0$)}, {\rm (2$_0$)}, and {\rm (5$_0$)} are empty conditions. 

For the inductive step,  assume that we have constructed $Y_m\colon \overline{\Mcal}\to\c^2$ for all $m\in\{0,\ldots,n-1\}$ meeting the above requirements for some $n\in \{1,\ldots, {\tt I}{\tt J}\}$. Let us find an immersion $Y_n$ satisfying properties {\rm (1$_n$)},$\ldots$,{\rm (6$_n$)}.

For the sake of simplicity, write $w_n:=\nu_\Dcal(p_{\sigma(n)}),$ and fix $u_n\in \lP w_n\rP^\bot \cap\s^3\subset T_{p_{\sigma(n)}}\Fr\Dcal$. Since $\{u_n,w_n\}$ is a $\lP\cdot,\cdot\rP$-orthonormal basis of $\c^2$, one has that
\begin{equation}\label{eq:Y_n-1}
Y_{n-1}=\lP Y_{n-1},u_n\rP u_n + \lP Y_{n-1},w_n\rP w_n.
\end{equation}

Recall that  $(\overline{\Mcal\setminus \Agot_{\sigma(n)}})\cap K_{\sigma(n)}=\emptyset$, and consider the holomorphic function $\phi\colon (\overline{\Mcal\setminus \Agot_{\sigma(n)}})\cup K_{\sigma(n)}\to \c$ given by
\begin{equation}\label{eq:phi}
\phi|_{\overline{\Mcal\setminus \Agot_{\sigma(n)}}}= \lP Y_{n-1},u_n\rP|_{\overline{\Mcal\setminus \Agot_{\sigma(n)}}} \quad\text{and}\quad \phi|_{K_{\sigma(n)}}= \zeta_n,
\end{equation}
where $\zeta_n\in\c$ is a constant with modulus large enough so that
\begin{equation}\label{eq:fuera}
(\zeta_nu_n+{\rm span}_\c(w_n)) \cap \overline{\Dcal'}=\emptyset.
\end{equation}
Such constant exists since $\overline{\Dcal'}$ is compact. Since $(\overline{\Mcal\setminus \Agot_{\sigma(n)}})\cup K_{\sigma(n)}$ is a Runge subset of a domain in $\Ncal$ containing $\overline{\Mcal}$, Runge's theorem furnishes a holomorphic function $\varphi\colon \overline{\Mcal}\to\c$ as close to $\phi$ as desired in the $\Ccal^1$ topology on $(\overline{\Mcal\setminus \Agot_{\sigma(n)}})\cup K_{\sigma(n)}$. 

\begin{claim}
If $\varphi$ is chosen close enough to $\phi$ in the $\Ccal^1$ topology on $(\overline{\Mcal\setminus \Agot_{\sigma(n)}})\cup K_{\sigma(n)}$, then the function $Y_n\colon \overline{\Mcal}\to \c^2$ given by
\begin{equation}\label{eq:Y_n}
Y_n:= \varphi u_n + \lP Y_{n-1},w_n\rP w_n
\end{equation}
satisfies properties {\rm (1$_n$)},$\ldots$,{\rm (6$_n$)}.
\end{claim}
Indeed, first of all observe that, up to slightly modifying $\varphi$, $Y_n$ can be assumed to be an immersion by a general position argument. Since $\varphi \approx \phi =\lP Y_{n-1},u_n\rP$ on $\overline{\Mcal\setminus \Agot_{\sigma(n)}}$, then $Y_n \approx Y_{n-1}$ on $\overline{\Mcal\setminus \Agot_{\sigma(n)}}$, and {\rm (1$_n$)} and {\rm (6$_n$)} hold (take into account \eqref{eq:phi}, \eqref{eq:Y_n}, \eqref{eq:Y_n-1},  and {\rm (6$_{n-1}$)}). Property {\rm (2$_n$)} directly follows from \eqref{eq:Y_n}, \eqref{eq:Y_n-1} and the definition of $u_n$ and $w_n$.

To check {\rm (3$_n$)} we distinguish two cases. If $a\neq n$, then $Y_n\approx Y_{n-1}$ on $ \overline{\Mcal\setminus \Agot_{\sigma(n)}} \supset \overline{\Agot_{\sigma(a)}\setminus K_{\sigma(a)}}$; hence {\rm (3$_{n-1}$)} implies that $Y_n(\overline{\Agot_{\sigma(a)}\setminus K_{\sigma(a)}})\subset \Tcal_{\sigma(a)}$. If $a=n$ then the inclusion $Y_n(\overline{\Agot_{\sigma(n)}\setminus K_{\sigma(n)}})\subset \Tcal_{\sigma(n)}$ is ensured by {\rm (2$_n$)}, {\rm (3$_{n-1}$)}, and the fact that $\Tcal_{\sigma(n)}$ is foliated by affine hyperplanes $\langle\cdot,\cdot\rangle$-orthogonal to $\nu_\Dcal(p_{\sigma(n)})$.

For {\rm (4$_n$)} we distinguish two cases again. If $a\neq n$, then {\rm (4$_{n-1}$)} and the fact that $Y_n\approx Y_{n-1}$ on $ \overline{\Mcal\setminus \Agot_{\sigma(n)}} \supset \overline{\beta_{\sigma(a)}\setminus K_{\sigma(a)}}$ give that $\pi_{\sigma(a)}(Y_n(\overline{\beta_{\sigma(a)}\setminus K_{\sigma(a)}}))\cap \pi_{\sigma(a)}(\overline{\Dcal'})=\emptyset$ as well. If $a= n$ then the assertion follows from {\rm (2$_n$)}, {\rm (4$_{n-1}$)}, and the definition of $\pi_{\sigma(n)}$.

Finally, property {\rm (5$_n$)} for $a<n$ is guaranteed by   {\rm (5$_{n-1}$)} and the fact that $Y_n\approx Y_{n-1}$ on $K_{\sigma(a)}$; whereas for $a=n$ is ensured by \eqref{eq:fuera} and that $\varphi\approx\phi$ on $K_{\sigma(n)}$.

This proves the claim, closes the induction, and concludes the construction of the immersions $Y_n\colon\overline{\Mcal}\to\c^2$, $n\in\{1,\ldots,{\tt I}{\tt J}\}$.

Let $\Scal$ denote the connected component of $Y_{{\tt I}{\tt J}}^{-1}(\Dcal')\subset \Mcal\Subset \Ncal$ containing $\overline{\Rcal}$; see {\rm (6$_{\tt IJ}$)}. Up to a slight deformation of $Y_{{\tt I}{\tt J}}$, assume that $\Scal\Subset\Ncal$ is a bordered domain. Define $Y:=Y_{{\tt I}{\tt J}}|_{\overline{\Scal}}\colon \overline{\Scal}\to\c^2$ and let us check that $Y$ meets all the requirements in the statement of the lemma.

Indeed, properties {\rm (x)} and {\rm (1$_n$)}, $n\in \{1,\ldots,{\tt I}{\tt J}\}$, give that 
\begin{equation}\label{eq:<xi}
\|Y_{{\tt I}{\tt J}}-X\|_{1,\overline{\Rcal};\vartheta_\Ncal}<\delta,
\end{equation}
proving Lemma \ref{lem:main}-{\rm (b)}. 

Properties {\rm (4$_{{\tt I}{\tt J}}$)} and {\rm (5$_{{\tt I}{\tt J}}$)} imply that $Y_{{\tt I}{\tt J}}(b\overline{\Mcal})\cap\overline{\Dcal'}=\emptyset$; observe that $b\overline{\Mcal}= \cup_{a=1}^{{\tt I}{\tt J}}\beta_{\sigma(a)}$. This property and the definition of $\overline{\Scal}$ ensure item {\rm (d)} in the lemma.

From {\rm (6$_{{\tt I}{\tt J}}$)} it follows that
\begin{equation}\label{eq:Y(R)}
Y(\overline{\Rcal})\subset \Dcal_\delta\Subset \Dcal',
\end{equation}
hence $\Rcal \Subset \Scal$ and Lemma \ref{lem:main}-{\rm (a)}  holds by the Maximum Principle. Furthermore, \eqref{eq:Y(R)} and {\rm (5$_{{\tt I}{\tt J}}$)} show that 
$b\overline{\Scal}\subset \overline{\Mcal}\setminus \big( \overline{\Rcal} \cup (\cup_{a=1}^{{\tt I}{\tt J}} K_{\sigma(a)}) \big)= \cup_{a=1}^{{\tt I}{\tt J}} \overline{\Agot_{\sigma(a)}\setminus K_{\sigma(a)}}$, and so $\overline{\Scal}\setminus{\Rcal}\subset \cup_{a=1}^{{\tt I}{\tt J}} \overline{\Agot_{\sigma(a)}\setminus K_{\sigma(a)}}$ as well. Then {\rm (3$_{{\tt I}{\tt J}}$)} gives that
\begin{equation}\label{eq:Y(S-R)}
Y(\overline{\Scal}\setminus{\Rcal})\subset \Tcal\cap\overline{\Dcal'} \subset \overline{\Dcal'}\setminus\overline{\Dcal}_{-\varepsilon}
\end{equation}
(take into account that $\Tcal$ has radius $\varepsilon$ for the latter inclusion),
proving Lemma \ref{lem:main}-{\rm (c)}. Finally, \eqref{eq:Y(R)} and \eqref{eq:Y(S-R)} guarantee item {\rm (e)}. 

This concludes the proof.
\end{proof}

%%%%%%%%%%%%%%%%%
%%%%%%%%%%%%%%%%%
%%%%%%%%%%%%%%%%%

\subsection{The desingularization lemma}\label{subsec:embe}

In this subsection we prove the following desingularization result for complex curves in $\c^2$; it is the third key in the proof of Lemma \ref{lem:compilation}.

\begin{lem}\label{lem:embe} 
Let $\Dcal\subset \c^2$ be a strictly convex bounded regular domain. Let $\Ncal$ be an open Riemann surface, let $\vartheta_\Ncal$ be a nowhere-vanishing holomorphic $1$-form on $\Ncal$, and let  $\Rcal$ and $ \Mcal$ be  bordered domains in $\Ncal$, $\Rcal \Subset \Mcal$. Let $X\colon \Ncal\to \c^2$ be a holomorphic immersion satisfying that 
\begin{enumerate}[\rm (I)]
\item $X(b \overline{\Mcal})\subset \Fr \Dcal$ (hence $X(\overline{\Rcal})\subset\Dcal$) and 
\item there are no double points of $X(\overline{\Mcal})$ in  $X(\overline{\Rcal})$; in particular,  $X|_{\overline{\Rcal}}$ is an embedding. 
\end{enumerate}

Then, for any $\epsilon>0$ there exist an open Riemann surface $\Wcal,$ a bordered domain $\Scal \Subset \Wcal,$ and a holomorphic embedding $F\colon\Wcal \to \c^2$ such that:
\begin{enumerate}[\rm (A)]
\item $\overline{\Rcal} \subset \Scal$.
\item $\|F-X\|_{1,\overline{\Rcal};\vartheta_\Ncal}<\epsilon$ and the Hausdorff distance $\dgot^{\rm H} \big( X(\overline{\Mcal}\setminus{\Rcal}),F(\overline{\Scal}\setminus{\Rcal}) \big)<\epsilon$. In particular, $\dgot^{\rm H} \big( X(\overline{\Mcal}),F(\overline{\Scal}) \big)<\epsilon$.
\item $F(b \overline{\Scal}) \subset \Fr \Dcal$.
%\item $F|_{\overline{\Scal}}$ is an embedding.
\end{enumerate}
\end{lem}

The proof of the lemma consists of replacing every normal crossing in $X(\overline{\Mcal})$ by an embedded annulus. It is important to point out that, although this surgery increases the topology,  the arising embedded complex curve $F(\overline{\Scal})$ contains a biholomorphic copy of $\overline{\Rcal}$, which is $\Ccal^1$ close to $X(\overline{\Rcal})$. 
%This will ensure Lemma \ref{lem:compilation}-{\em i)}, which has allowed us to have limit in the proof of Theorem \ref{th:embe}; see Sec.\ \ref{sec:theorem}.

Roughly speaking, we take a holomorphic defining function $\Pscr_0 \colon \overline{\Dcal} \to \c$ of $X(\overline{\Mcal})$ so that  $X(\overline{\Mcal})\equiv\{(\zeta,\xi) \in \overline{\Dcal}\colon \Pscr_0(\zeta,\xi)=0\}$. Then we take a nearby smooth level set 
$\Cscr_\lambda:=\{(\zeta,\xi) \in \overline{\Dcal}\colon \Pscr_0(\zeta,\xi)=\lambda\}$, $\lambda$ close to $0$.  If $\lambda$ is close enough to $0$, $\Cscr_\lambda$ is an embedded complex curve containing  a biholomorphic copy of $\overline{\Rcal}$, and the surface $F(\overline{\Scal}):=\Cscr_\lambda$ solves the lemma.

\begin{proof}[Proof of Lemma \ref{lem:embe}] 

Let  $\Mcal'\Subset \Ncal$ be a bordered domain such that $\Mcal \Subset \Mcal'$,  
\begin{equation}\label{eq:XbM'}
\text{$X(b \overline{\Mcal'})\cap \overline{\Dcal}=\emptyset$, and there are no double points of $X(\overline{\Mcal'})$ in  $X(\overline{\Rcal})$;}
\end{equation}
take into account properties {\rm (I)} and {\rm (II)}.

Let $F_0\colon \overline{\Mcal'} \to \c^2$ be a slight deformation of $X\colon \overline{\Mcal'} \to \c^2$ so that:
\begin{enumerate}[{\rm (i)}]
\item ${F_0}\colon \overline{\Mcal'} \to \c^2$ is a holomorphic immersion.
\item ${F_0}(\overline{\Rcal})\subset \Dcal$, $F_0(b \overline{\Mcal'})\cap \overline{\Dcal}=\emptyset$ (see  \eqref{eq:XbM'}), and ${F_0}(\Mcal')$  and $\Fr \Dcal$ meet transversally.
\item ${F_0}$ is as close to $X$ as desired in the $\Ccal^1$ topology on $\overline{\Mcal'}$; in particular 
\begin{itemize}
\item $\|{F_0}-X\|_{1,\overline{\Rcal};\vartheta_\Ncal}<\epsilon/2$, 
\item there are no double points of $F_0(\overline{\Mcal'})$ in  $F_0(\overline{\Rcal})$ (in particular,  ${F_0}|_{\overline{\Rcal}}:\overline{\Rcal}\to \c^2$ is an embedding), and
\item $\dgot^{{\rm H}}(X(\overline{\Mcal}\setminus{\Rcal}),{F_0}(\overline{S}_0\setminus{\Rcal}))<\epsilon/2$, where $S_0 \Subset \Mcal'$ is the connected component of ${F_0}^{-1}(\Dcal)$ containing $\overline{\Rcal}$.
\end{itemize}
\item All the double points of $F_0(\overline{S}_0)$ are normal crossings and lie in $\Dcal$.
\end{enumerate}
Take into account Remark \ref{rem:position}. Denote by $\Acal:=\big\{ \{P,P^*\}\subset \overline{S}_0\colon P\neq P^* \;\text{and} \; F_0(P)=F_0(P^*)\big\}$ the (finite) {\em double points set of $F_0|_{\overline{S}_0}$}, and call $F_0(\Acal):=\{F_0(P)\colon \{P,{P^*}\}\in \Acal\}\subset\c^2$. Notice from {\rm (ii)} and  {\rm (iii)} that  
\begin{equation}\label{eq:posicion}
\text{$F_0(\overline{S}_0)$ and $\Fr \Dcal$ meet transversally and $F_0(\Acal) \cap \big(F_0(\overline{\Rcal}) \cup \Fr \overline{\Dcal}\big)=\emptyset$.}
\end{equation}
 Without loss of generality, $\overline{S}_0$ can be assumed to be homeomorphic to $\overline{\Mcal}$, but not biholomorphic. 
%Notice also that all the double points of  $F_0(\overline{S}_0)$ are normal crossings. 

The domain $\Dcal$ is a Stein manifold whose second cohomology group $\Hcal^2(\overline{\Dcal},\z)$ vanishes. This implies that  any divisor in $\overline{\Dcal}$ is principal (see for instance \cite[p.\ 98]{Remmert-book}), hence there exists a holomorphic function 
$\Pscr_0\colon \overline{\Dcal} \to \c$   such that 
\[
F_0(\overline{S}_0)=\{(\zeta,\xi) \in \overline{\Dcal}\colon \Pscr_0(\zeta,\xi)=0\}.
\]

From {\rm (iv)} and the fact that $F_0$ is an immersion,  it is not hard to check that $q\in F_0(\Acal_0)$ if and only if
\begin{equation} \label{eq:removal}
\frac{\partial \Pscr_0}{\partial \zeta}(q)=\frac{\partial \Pscr_0}{\partial \xi}(q)=\Pscr_0(q)=0 \quad \text{and} \quad {\rm H}(\Pscr_0)_{q}\neq 0,
\end{equation}
where ${\rm H}(\Pscr_0)_{q}$ denotes the Hessian of $\Pscr_0$ at $q$.

The next step of the proof consists of removing from  $F_0(\overline{S}_0)$ the normal crossings. To do this, we deform this curve in an appropriate way. For each $\lambda \in  \c\setminus\{0\}$ consider the holomorphic function 
\[
\Pscr_\lambda\colon \overline{\Dcal}\to \c, \quad \Pscr_\lambda(\zeta,\xi):=\Pscr_0(\zeta,\xi)-\lambda,
\]
and denote by
\[
\Scal_\lambda:=\{(\zeta,\xi)\in {\Dcal}\colon \Pscr_\lambda(\zeta,\xi)=0\}.
\]
Obviously, 
\begin{equation}\label{eq:lambda0}
\lim_{\lambda\to 0}\Pscr_\lambda= {\Pscr_0}\quad\text{uniformly on  $\c^2$}.
\end{equation} 
\begin{claim} \label{cla:final}
If $|\lambda|>0$ is small enough, 
there exists an  open {\em embedded} complex curve $\Cscr_\lambda$ in $\c^2$ such that $\Cscr_\lambda$ and $\Fr\Dcal$ meet transversally and $\Cscr_\lambda\cap \overline{\Dcal}=\overline{\Scal}_\lambda$.
\end{claim}
\begin{proof} To prove the claim, it suffices to check that $0$ is a regular value for $\Pscr_\lambda|_{\overline{\Dcal}}$. 

Consider the holomorphic function $f\colon \overline{\Dcal} \times \c \to \c^3$ given by: 
\[
f(p,\lambda)=\Big(\frac{\partial\Pscr_0}{\partial \zeta}\,,\, \frac{\partial\Pscr_0}{\partial \xi} \,,\, \Pscr_\lambda\Big)(p).
\] 
 Obviously, $0$ is a regular value for $\Pscr_\lambda|_{\overline{\Dcal}}$ if and only if $f^{-1}(0,0,0)\subset S_0$ (take into account that $\Scal_\lambda \cap F_0(S_0)=\emptyset$, $\lambda \neq 0$). Since any double point $p$ of $\Scal_\lambda$ satisfies 
$\frac{\partial \Pscr_0}{\partial \zeta}(p)=\frac{\partial \Pscr_0}{\partial \xi}(p)=0$, equations \eqref{eq:posicion},  \eqref{eq:removal}, and \eqref{eq:lambda0} give that the double points set of $\Scal_\lambda$  converges, as $\lambda$ goes to $0$, to $F_0(\Acal)$. 
On the other hand, the Jacobian  of $f$
\[
{\rm Jac}f_{(q,0)}=-{\rm H}(\Pscr_0)_q\neq 0 \quad\text{for any $ q\in F_0(\Acal)$};
\]
see {\rm (iv)} and \eqref{eq:removal}.  Therefore,  $f$ is local biholomorphism around points $(q,0)$, $q\in F_0(\Acal)$, and we are done.

The claim follows from \eqref{eq:posicion}, \eqref{eq:lambda0} and the fact that $\Scal_\lambda$ is a submanifold of $\overline{\Dcal}$.
\end{proof}

As a consequence of Claim \ref{cla:final}, the embedded complex curve $\Scal_\lambda$ is a (connected) bordered domain in $\Cscr_\lambda$  with $b \overline{\Scal}_\lambda \subset \Fr \Dcal$.

On the other hand, one has that
\begin{equation}\label{eq:lambda1}
\text{$\lim_{\lambda\to 0} \dgot^{\rm H} \big(\overline{\Scal}_\lambda\cap \Kcal,F_0(\overline{S}_0)\cap \Kcal\big)=0\;$ for any compact $\Kcal \subset \overline{\Dcal}$.}
\end{equation}
It is interesting to notice that the convergence of $\overline{\Scal}_\lambda$ to $F_0(\overline{S}_0)$, as $\lambda$ goes to $0$, is nice outside the double points set $F_0(\Acal)$, as the following claim shows:

\begin{claim} \label{cla:nice} 
Let $\Omega \Subset S_0$ be a bordered domain such that 
$F_0(\overline{\Omega})\cap F_0(\Acal)=\emptyset$ (in particular,  ${F_0}|_{\overline{\Omega}}\colon \overline{\Omega} \to \c^2$ is an embedding). Then, if  $|\lambda|>0$ is small enough, there exist a bordered domain $\Omega_\lambda \Subset \Scal_\lambda$ and a biholomorphism $\sigma_\lambda \colon \overline{\Omega} \to \overline{\Omega}_\lambda$ such that
\[
\lim_{\lambda \to 0} \|\sigma_\lambda -F_0\|_{1,\overline{\Omega};\vartheta_\Ncal}=0.
\]
\end{claim}
\begin{proof} 
Write ${F_0}=(z_0,w_0)$ and choose any holomorphic $G:=(f_1,f_2)\colon\overline{\Omega} \to \c^2$  such that 
\begin{equation} \label{eq:G}
\text{$f_2 dz_0-f_1 dw_0$ vanishes nowhere on $\overline{\Omega}$};
\end{equation}
existence of such a $G$ follows from the fact that ${F_0}$ is an immersion on $\overline{\Omega}$ and  Riemann-Roch's theorem.  For any $\delta>0$, denote by $\d_\delta=\{t \in \c\colon |t|<\delta\}$ and set the holomorphic function
\[
\Phi\colon \overline{\Omega} \times \d_1 \to  \c^2, \quad  \Phi(P,t)= {F_0}(P)+t G(P).
\]
Notice that $\Phi$ is a local biholomorphism around  $(P,0),$ $P \in \Omega$ (see \eqref{eq:G}). Denote by
$V_\delta=\Phi(\overline{\Omega} \times \d_\delta)$, $\delta\in ]0,1[$, and choose $\delta$ small enough so that $\overline{V}_\delta \subset \Dcal$, $V_\delta\cap F_0(\Acal)=\emptyset$, and
$$\Psi \colon \overline{\Omega} \times \d_\delta \to  V_\delta, \quad \Psi(P,t)=\Phi(P,t),$$ is a biholomorphism; take into account that ${F_0}|_{\overline{\Omega}}\colon \overline{\Omega}\to \c^2$ is an embedding and $F_0(\overline{\Omega})\cap F_0(\Acal)=\emptyset$. Call $\pi\colon \overline{\Omega}\times \d_1 \to \overline{\Omega}$ the natural holomorphic projection.

If $\delta$ is small enough,   $0$ is a regular value for $\Pscr_\lambda|_{V_\delta}$ for any $\lambda$; take into account \eqref{eq:removal} and the fact that  $F_0(\overline{\Omega})\cap F_0(\Acal)=\emptyset$. Therefore, 
$\Gamma:=\{\Scal_\lambda\cap V_\delta\colon \lambda \in \c\}$ is a regular holomorphic foliation of $V_\delta$ transverse to the field $G\circ \pi\circ \Psi^{-1}$ (see \eqref{eq:G}), and so,  $\pi$ is one to one on sheets of $\Gamma$. To finish, it suffices to set  $\Omega_\lambda:= V_\delta\cap{\Scal}_\lambda$ and observe that for   $|\lambda|>0$ small enough:
\begin{itemize}
\item ${\Omega}_\lambda \Subset \Scal_\lambda$ and $\rho_\lambda:=(\pi \circ \Psi^{-1})|_{\overline{\Omega}_\lambda}\colon \overline{\Omega}_\lambda \to \overline{\Omega}$ is  a biholomorphism, and
\item $\lim_{\lambda \to 0} \|\sigma_\lambda -F_0\|_{1,\overline{\Omega};\vartheta_\Ncal}=0$, where $\sigma_\lambda:=\rho_\lambda^{-1}$;
\end{itemize}
see \eqref{eq:lambda0}. This proves the claim.
\end{proof}

In view of Claim \ref{cla:final}, to finish it suffices to find a bordered domain $\Rcal_\lambda \Subset \Scal_\lambda\Subset\Cscr_\lambda$  biholomorphic to $\Rcal$ such that   $\overline{\Rcal}_\lambda$ converge as $\lambda \to 0$ to $F_0(\overline{\Rcal})$; see  \eqref{eq:jauar} below.

Indeed, Claim \ref{cla:nice} applies to $\Rcal$ furnishing a bordered domain $\Rcal_\lambda\Subset \Scal_\lambda$ and a biholomorphism $\sigma_\lambda\colon \overline{\Rcal}\to \overline{\Rcal}_\lambda $, $|\lambda|>0$ small enough. Furthermore, if $\lambda_0\in \c\setminus\{0\}$ is sufficiently close to $0$, the following conditions are satisfied:
\begin{itemize}
\item  $\sigma_{\lambda_0}\colon \overline{\Rcal} \to \overline{\Rcal}_{\lambda_0}$ is a biholomorphism.
\item $\|\sigma_{\lambda_0}-{F_0}\|_{1,\overline{\Rcal};\vartheta_\Ncal}<\epsilon/2$.
\item $\dgot^{\rm H} \big( {F_0}(\overline{S}_0\setminus {\Rcal}),\overline{\Scal}_{\lambda_0}\setminus {\Rcal}_{\lambda_0} \big)<\epsilon/2$.
\end{itemize}
For the last item, take into account that $F_0(\Acal) \cap F_0(\overline{\Rcal})=\emptyset$ (see \eqref{eq:posicion}), \eqref{eq:lambda1}, and 
\begin{equation}\label{eq:jauar}
 \lim_{\lambda \to 0} \|\sigma_{\lambda}-{F_0}\|_{1,\overline{\Rcal};\vartheta_\Ncal}=0.
\end{equation}

Set $\Scal:=\Scal_{\lambda_0}$ and $\Wcal=\Cscr_{\lambda_0}$. Up to identifying $\overline{\Rcal}$ with $\overline{\Rcal}_{\lambda_0}$ via $\sigma_{\lambda_0}$ (hence $\overline{\Rcal}\subset \Scal$) and taking into account {\rm (iii)} and Claim \ref{cla:final}, the open Riemann surface $\Wcal,$ the bordered domain $\Scal \Subset \Wcal,$ and the holomorphic embedding $F:={\rm Id}\colon\Wcal \to \Wcal \hookrightarrow \c^2$ satisfy all the requirements in the statement of the lemma. 
\end{proof}

%%%%%%%%%%%%%%%%%
%%%%%%%%%%%%%%%%%
%%%%%%%%%%%%%%%%%

\subsection{Proof of Lemma \ref{lem:compilation}}\label{subsec:proofcompi}

By \eqref{eq:compilation1},  $X(\overline{\Ucal})$ and $\Fr \Dcal$ meet transversally (see Remark \ref{re:trans}). Thus, we can find a small $\rho\in]0,\epsilon/2[$ and a bordered domain $\Vcal\Subset\Ncal$ such that $\Dcal_\rho\Subset \Dcal'$,
$\Ucal\Subset\Vcal$, $X$ extends as a holomorphic embedding $X\colon\overline{\Vcal}\to\c^2$, $X(b\overline{\Vcal})\subset\Fr\Dcal_{\rho}$, $X(\overline{\Vcal}\setminus\Ucal) \subset \overline{\Dcal}_\rho\setminus\Dcal$, and 
\begin{equation} \label{eq:d-d}
|{\bf d}(\Dcal_\rho,\Fr\Dcal')-{\bf d}(\Dcal,\Fr\Dcal')|<\epsilon/2.
\end{equation}

Take  $\epsilon_0\in ]0,\rho/2[$, and notice that
\begin{equation}\label{eq:U}
X(\overline{\Ucal})\subset \overline{\Dcal}\subset \Dcal_{\rho-\epsilon_0}; 
\end{equation}
see  \eqref{eq:compilation1} and use the Maximum Principle.
Since  $X(b\overline{\Vcal})\subset\Fr\Dcal_{\rho}$, Lemma \ref{lem:net} furnishes a tangent net $\Tcal$ of radius $\mu\in
 ]0,\min \{\epsilon_0,\dist(\overline{\Dcal}_\rho,\Fr\Dcal'),1/\kappa(\Dcal_\rho)\}[$  for $\Dcal_\rho$ such that:
\begin{enumerate}[\rm ({A}1)]
\item $X(b\overline{\Vcal})\subset\Tcal$, and
\item $\ell(\alpha)>{\bf d}(\Dcal_\rho,\Fr\Dcal')-\mu$ for any Jordan arc $\alpha$ in $\Tcal$ connecting $\Fr\Dcal_\rho$ and $\Fr\Dcal'$.
\end{enumerate}

Take $\varsigma\in]0,\mu[$ small enough so that $\Dcal_{\rho+\varsigma}\Subset \Dcal'$,
\begin{enumerate}[\rm ({B}1)]
\item $\ell(\alpha)>{\bf d}(\Dcal_\rho,\Fr\Dcal')-\mu$ for any Jordan arc $\alpha$ in $\Tcal$ connecting $\Fr\Dcal_{\rho+\varsigma}$ and $\Fr\Dcal'$ (see {\rm (A2)}), and
\item any holomorphic map $G\colon\overline{\Vcal}\to\c^2$ with $\|G-X\|_{1,\overline{\Vcal};\vartheta_\Ncal}<\varsigma$ satisfies that
\begin{enumerate}[\rm ({B2}.1)]
\item $G$ is an embedding in $\overline{\Vcal}$ (recall that $X\colon\overline{\Vcal}\to\c^2$ is an embedding and use the Cauchy estimates),
\item $G(\overline{\Ucal})\subset\Dcal_{\rho-\epsilon_0}$, and $G(\overline{\Vcal}\setminus\Ucal)\cap \overline{\Dcal}_{-\epsilon}=\emptyset$ (see \eqref{eq:U} and use the fact $X(\overline{\Vcal}\setminus\Ucal) \subset \overline{\Dcal}_\rho\setminus\Dcal$ is disjoint from $\overline{\Dcal}_{-\epsilon}$).
\end{enumerate}
\end{enumerate}

From {\rm (A1)} and \eqref{eq:compilation1}, Lemma \ref{lem:main} applies to the data
\[
(\Dcal,\Dcal',\varepsilon,\Tcal,\delta,\Ncal,\vartheta_\Ncal,\Rcal,X)=(\Dcal_\rho,\Dcal',\mu,\Tcal,\varsigma,\Ncal,\vartheta_\Ncal,\Vcal,X)
\]
providing a bordered domain $\Wcal\Subset\Ncal$ and a holomorphic immersion $Y\colon \overline{\Wcal}\to\c^2$ such that: 
\begin{enumerate}[\rm ({C}1)]
\item $\Vcal\Subset\Wcal$ and $\Vcal$ and $\Wcal$ are homeomorphically isotopic.
\item $\|Y-X\|_{1,\overline{\Vcal};\vartheta_\Ncal}<\varsigma$; in particular, $Y|_{\overline{\Vcal}}$ is an embedding (see {\rm (B2)}).
\item $Y(\overline{\Wcal}\setminus\Vcal)\subset\overline{\Dcal'}\setminus\overline{\Dcal}_{\rho-\mu}$.
\item $Y(b\overline{\Wcal})\subset \Fr\Dcal'$.
\item $Y(\overline{\Wcal})\subset\Dcal_{\rho+\varsigma}\cup\Tcal$.
\end{enumerate}

Notice that 
\begin{equation}\label{eq:U1}
Y(\overline{\Ucal})\subset\Dcal_{\rho-\epsilon_0}\quad\text{and}\quad Y(\overline{\Vcal}\setminus\Ucal)\cap \overline{\Dcal}_{-\epsilon}=\emptyset;
\end{equation}
take into account {\rm (C2)} and {\rm (B2.2)}. Since $\mu<\epsilon_0<\rho$,  {\rm (C3)} and the latter assertion in \eqref{eq:U1} give that
\begin{equation}\label{eq:tiron}
Y(\overline{\Wcal}\setminus \Ucal)\cap\overline{\Dcal}_{-\epsilon}=\emptyset.
\end{equation}

The fact that $Y|_{\overline{\Vcal}}$ is an embedding (see {\rm (C2)}), property {\rm (C3)}, the first assertion in \eqref{eq:U1}, and the fact $\mu<\epsilon_0$, ensure that there are no double points of $Y(\overline{\Wcal})$ in $Y(\overline{\Ucal})$. From this fact and {\rm (C4)}, Lemma \ref{lem:embe} applies to the data
\[
(\Dcal,\Ncal,\vartheta_\Ncal,\Rcal,\Mcal,X,\epsilon)=(\Dcal',\Ncal,\vartheta_\Ncal, \Ucal, \Wcal, Y,\eta),
\]
where $\eta\in]0,\epsilon-\varsigma[$ will be specified later, furnishing an open Riemann surface $\Ncal'$, a bordered domain $\Ucal'$, and a holomorphic embedding $F\colon \Ncal'\to\c^2$ satisfying:
\begin{enumerate}[\rm ({D}1)]
\item $\overline{\Ucal}\Subset \Ucal'$.
\item $\|F-Y\|_{1,\overline{\Ucal};\vartheta_\Ncal}<\eta$ and $\dgot^{\rm H}(Y(\overline{\Wcal}\setminus \Ucal),F(\overline{\Ucal'}\setminus \Ucal))<\eta$.
\item $F(b\overline{\Ucal'})\subset\Fr\Dcal'$.
%\item $X':=F|_{\overline{\Ucal'}}$ is an embedding.
\end{enumerate}

Let us check that the embedding $X':=F|_{\overline{\Ucal'}}\colon\overline{\Ucal'}\to\c^2$ solves the lemma. {\rm (D1)} and {\rm (D3)} agree with Lemma \ref{lem:compilation}-{\it i)} and {\it iii)}, respectively. Property {\it ii)} follows from {\rm (C2)} and {\rm (D2)}; recall that $\eta<\epsilon-\varsigma$. Property {\it iv)} is given by \eqref{eq:tiron} and {\rm (D2)} provided that $\eta$ is chosen small enough.   
 
Finally, let us check {\it v)}. Let $\gamma$ be a Jordan arc in $\overline{\Ucal'}$ connecting $b\overline{\Ucal}$ and $b\overline{\Ucal'}$. From {\rm (C5)}, {\rm (D2)}, and the first assertion in \eqref{eq:U1}, it follows that $X'(\overline{\Ucal'})\subset \Dcal_{\rho+\varsigma}\cup\Tcal$ and $X'(\overline{\Ucal})\subset\Dcal_{\rho-\epsilon_0}\Subset\Dcal_{\rho+\varsigma}$,   provided that $\eta$ is small enough. Taking also {\rm (D3)} into account, we deduce that $\gamma$ contains a sub-arc $\gamma'$ such that $X'(\gamma')$ is contained in $\Tcal$ and connects $\Fr\Dcal_{\rho+\varsigma}$ and $\Fr\Dcal'$. By {\rm (B1)}, $\ell(X'(\gamma))\geq \ell(X'(\gamma'))>{\bf d}(\Dcal_\rho,\Fr\Dcal')-\mu> {\bf d}(\Dcal,\Fr\Dcal')-\epsilon$. For the last inequality,  take into account that $\mu<\epsilon/2$ and 
 ${\bf d}(\Dcal_\rho,\Fr\Dcal')>{\bf d}(\Dcal,\Fr\Dcal')-\epsilon/2$; see \eqref{eq:d-d}. This concludes the proof.

%%%%%%%%%%%%%%%%%
%%%%%%%%%%%%%%%%%
%%%%%%%%%%%%%%%%%

\section{Image complete complex curves in convex domains}\label{sec:extrinsic}

In this section we make use of Lemmas \ref{lem:net} and \ref{lem:main} in order to prove Theorem \ref{th:extrinsic} below. Observe that Theorem \ref{th:intro2} in the introduction is a particular instance of it.

Let $\Ncal$ be an open Riemann surface.  A domain $U\subset \Ncal$ is said to be {\em homeomorpically isotopic} to $\Ncal$ if there exists a homeomorphism $\mu\colon U \to \Ncal$ satisfying $\mu_*=i_*,$ where $i\colon U \hookrightarrow \Ncal$ is the inclusion map and 
$\mu_*,$ $i_*\colon \Hcal_1(U,\z) \to \Hcal_1(\Ncal,\z)$ are the induced group morphisms. In this case, $\Hcal_1(U,\z)$ and $\Hcal_1(\Ncal,\z)$ will be identified via $\mu_*.$

\begin{thm}\label{th:extrinsic}
Let $\Bcal$ be a (possibly neither bounded nor regular) convex domain in $\c^2$ and let $\Dcal\Subset\Bcal$ be a bounded regular strictly convex domain. Let $\Ncal$ be an open Riemann surface equipped with a  nowhere-vanishing holomorphic $1$-form $\vartheta_\Ncal$, let $\Mcal\Subset \Ncal$ be a Runge bordered domain,  and let $X\colon \overline{\Mcal}\to\c^2$ be a holomorphic immersion such that 
\begin{equation}\label{eq:th-extrinsic}
X(b\overline{\Mcal})\subset \Fr \Dcal.
\end{equation}

Then, for any $\epsilon\in]0,\min \{\dist(\overline{\Dcal},\Fr\Bcal),1/\kappa(\Dcal)\}[$ there exist a domain $U\subset\Ncal$ and a holomorphic immersion $Y\colon U\to\c^2$ satisfying the following properties:
\begin{enumerate}[\rm (A)]
\item $\Mcal\Subset U$ and $U$ is homeomorphically isotopic to $\Ncal$.
\item $\|Y-X\|_{1,\overline{\Mcal};\vartheta_\Ncal}<\epsilon$ (see \eqref{eq:norma1}).
\item $Y(U)\subset \Bcal$ and $Y\colon U\to \Bcal$ is a proper map.
\item $Y(U\setminus\overline{\Mcal})\subset \Bcal\setminus \Dcal_{-\epsilon}$.
\item $Y$ is image complete  (see Def.\ \ref{def:extrinsic}).
\end{enumerate}
\end{thm}

\begin{proof}
Denote by $\Dcal^0:=\Dcal$ and let $\{\Dcal^n\}_{n\in\n}$ be a ${\bf d}$-proper sequence of convex domains in $\Bcal$ with $\Dcal^0\Subset\Dcal^1$; see Def.\ \ref{def:proper} and Lemma \ref{lem:sucprop}.

Call $N_0:=\Mcal$ and let $\{N_n\}_{n\in\n}$ be an exhaustion of $\Ncal$ by bordered domains so that $\overline{N}_n\subset\Ncal$ is Runge, $N_{n-1}\Subset N_n$ and the Euler characteristic $\chi(\overline{N}_n\setminus N_{n-1})\in\{-1,0\}$ for all $n\in\n$; cf. \cite[Lemma 4.2]{AL-Narasimhan}. 

Call $U_0:=N_0$, $X_0:=X$, and $\eta_0:={\rm Id}_{U_0}\colon U_0\to U_0$, let $\epsilon_0\in]0,\epsilon/2[$, and let us construct a sequence $\{\Upsilon_n=(U_n, \eta_n, X_n, \epsilon_n)\}_{n\in\n}$; where
\begin{itemize}
\item $U_n\Subset\Ncal$ is a bordered domain and  $\overline{U}_n$ is Runge in $\Ncal$,
\item $\eta_n\colon \overline{U}_n\to \overline{N}_n$ is an isotopical homeomorphism,
\item $X_n\colon \overline{U}_n\to\c^2$ is a holomorphic immersion, and
\item $\epsilon_n>0$, 
\end{itemize}
such that the following properties hold for all $n\in\n$:
\begin{enumerate}[\rm (1$_n$)]
\item $U_{n-1}\Subset U_n$.
\item $\eta_n|_{\overline{U}_{n-1}}=\eta_{n-1}$.
\item $\epsilon_n$ is a positive real number satisfying that
\begin{itemize}
\item $\epsilon_n<\min\{ \epsilon_{n-1}/2, 1/\kappa(\Dcal^{n-1}), \dist(\Dcal^{n-1},\Fr\Dcal^n)\}(<\epsilon/2^{n+1})$ and 
\item any holomorphic function $G\colon \overline{U}_{n-1}\to\c^2$ with $\|G-X_{n-1}\|_{1,\overline{U}_{n-1};\vartheta_\Ncal}<2\epsilon_n$ is an immersion.
\end{itemize}
%\begin{equation}\label{eq:rhoN}
%\rho_n:=\min\big\{\min_{\overline{U}_a} \|dX_a/\vartheta_\Ncal\|\colon a\in\{0,\ldots,n-1\}\big\}>0.
%\end{equation}
\item $\|X_n-X_{n-1}\|_{1,\overline{U}_{n-1};\vartheta_\Ncal}<\epsilon_n$.
\item $X_n(\overline{U}_a\setminus U_{a-1})\subset \Dcal^{a+1}\setminus \overline{\Dcal^{a-1}_{-\epsilon_a}}$ for all $a\in\{1,\ldots,n\}$.
\item $X_n(b\overline{U}_n)\subset \Fr\Dcal^n$; hence $X_n(\overline{U}_n\setminus{U}_{n-1})\subset \overline{\Dcal^{n}}\setminus \overline{\Dcal^{n-1}_{-\epsilon_n}}$.
\item $\ell(\gamma)> {\bf d}(\Dcal^{a-1},\Fr\Dcal^a)-\epsilon_a$ for any Jordan arc $\gamma\subset X_n(\overline{U}_n)\subset\c^2$ connecting $\Fr\Dcal^{a-1}$ and $\Fr\Dcal^a$, for all $a\in\{1,\ldots,n\}$.
\end{enumerate}

The sequence will be constructed in an recursive way. For the basis of the induction take $\Upsilon_0=(U_0,\eta_0,X_0,\epsilon_0).$ Notice that (6$_0$) agrees with \eqref{eq:th-extrinsic}, and the remaining properties ($j_0$), $j\neq 6,$ are empty.

%we choose any $\epsilon_1>0$ satisfying {\rm (3$_1$)}. From \eqref{eq:th-extrinsic} and Lemma \ref{lem:net}, there exist a tangent net $\Tcal_1$ of radius $<\epsilon_1$ for $\Dcal^0$ and a positive $\delta_1<\epsilon_1$, such that
%\begin{enumerate}[\rm (i)]
%\item $X_0(b\overline{U}_0)\Subset \Tcal_1$ and
%\item $\ell(\gamma)> {\bf d}(\Dcal^0,\Fr\Dcal^1)-\epsilon_1$ for any Jordan arc $\gamma\subset \Tcal_1$ connecting $\Fr\Dcal^0_{\delta_1}$ and $\Fr\Dcal^1$.
%\end{enumerate}
%
%Property {\rm (i)} allows one to apply Lemma \ref{lem:main} to the data
%\[
%\Dcal= \Dcal^0,\enskip \Dcal'=\Dcal^1,\enskip \varepsilon=\epsilon_1,\enskip \Tcal=\Tcal_1, \enskip \delta=\delta_1,\enskip \Rcal= U_0 ,\enskip\text{and}\enskip X=X_0.
%\]
%The resulting bordered domain $U_1$ and holomorphic immersion $X_1\colon\overline{U}_1\to\c^2$ meet properties {\rm (1$_1$)} and {\rm (4$_1$)}-{\rm (7$_1$)}. Indeed, all that properties follow directly except for {\rm (7$_1$)}, which is implied by {\rm (ii)} and Lemma \ref{lem:main}-{\rm (e)}. To finish the construction of $\Upsilon_1$ we choose any isotopical homeomorphism $\eta_1\colon \overline{U}_1\to \overline{N}_1$ satisfying {\rm (2$_1$)}, which exists by \eqref{eq:n1-n0}.  

For the inductive step, fix $n\in\n$ and  assume that we have already constructed $\Upsilon_m$  satisfying the above properties for all  $m\in\{0,\ldots, n-1\}$. Let us construct $\Upsilon_n$.

Choose any $\epsilon_n>0$ satisfying {\rm (3$_n$)} and
\begin{enumerate}[\rm (i)]
\item $\ell(\gamma)> {\bf d}(\Dcal^{n-2},\Fr\Dcal^{n-1})-\epsilon_{n-1}$ for any Jordan arc $\gamma$ in $X_{n-1}(\overline{U}_{n-1})$ connecting $\Fr\Dcal^{n-2}$ and $\Fr\Dcal^{n-1}_{-\epsilon_n}$; take into account {\rm (7$_{n-1}$)}. When $n=1$, this condition is empty.
\end{enumerate}
Such $\epsilon_n$ exists since $X_{n-1} \colon \overline{U}_{n-1}\to\c^2$ is an immersion.

 We distinguish two cases.

\noindent$\bullet$ Assume that $\chi(\overline{N}_n\setminus N_{n-1})=0$. From {\rm (6$_{n-1}$)} and Lemma \ref{lem:net}, there exists a tangent net $\Tcal_n$ of radius $<\epsilon_n$ for $\Dcal^{n-1}$ such that
\begin{enumerate}[\rm (i)]
\item[\rm (ii)] $X_{n-1}(b\overline{U}_{n-1})\Subset \Tcal_n$ and
\item[\rm (iii)] $\ell(\gamma)> {\bf d}(\Dcal^{n-1},\Fr\Dcal^n)-\epsilon_n$ for any Jordan arc $\gamma\subset \Tcal_n$ connecting $\Fr\Dcal^{n-1}$ and $\Fr\Dcal^n$.
\end{enumerate}

Let $\delta_n\in ]0,\epsilon_n[$ to be specified later, and  choose it small enough so that 

\begin{enumerate}[\rm (i)]
\item[{\rm (iv)}] $\ell(\gamma)> {\bf d}(\Dcal^{n-1},\Fr\Dcal^n)-\epsilon_n$ for any Jordan arc $\gamma\subset \Tcal_n$ connecting $\Fr\Dcal^{n-1}_{\delta_n}$ and $\Fr\Dcal^n$; see {\rm (iii)}.
\end{enumerate}

By properties {\rm (ii)} and (6$_{n-1}$), one can apply Lemma \ref{lem:main} to the data
\[
\Dcal= \Dcal^{n-1},\enskip \Dcal'=\Dcal^n,\enskip \varepsilon=\epsilon_n,\enskip \Tcal=\Tcal_n, \enskip \delta=\delta_n,\enskip \Rcal= U_{n-1} ,\enskip\text{and}\enskip X=X_{n-1}.
\]
The bordered domain $U_n$ (which is Runge since $\overline{U}_{n-1}$ is) and holomorphic immersion $X_n\colon\overline{U}_n\to\c^2$  furnished by Lemma \ref{lem:main} enjoy the properties {\rm (1$_n$)} and {\rm (4$_n$)}-{\rm (7$_n$)}.  Indeed, properties {\rm (1$_n$)}, {\rm (4$_n$)}, and {\rm (6$_n$)} follow straightforwardly. 

Property {\rm (5$_n$)} for $a=n$ is given by Lemma \ref{lem:main}-{\rm (c)}, whereas for $a<n$ it is ensured by {\rm (5$_{n-1}$)}  and Lemma \ref{lem:main}-{\rm (b)} provided that $\delta_n$ is small enough.  

Property {\rm (7$_n$)} for $a=n$ follows from Lemma \ref{lem:main}-{\rm (e)} and {\rm (iv)}; for $a=n-1$ is guaranteed by {\rm (i)} and Lemma \ref{lem:main}-{\rm (b)},{\rm (c)} provided that $\delta_n$ is chosen small enough; and for $a<n-1$ by {\rm (7$_{n-1}$)}  and Lemma \ref{lem:main}-{\rm (b)} provided that $\delta_n$ is small enough. 

Finally we choose any isotopical homeomorphism $\eta_n\colon \overline{U}_n\to \overline{N}_n$ satisfying {\rm (2$_n$)}; such exists since $\chi(\overline{N}_n\setminus N_{n-1})=0=\chi(\overline{U}_n\setminus U_{n-1})$.

\noindent$\bullet$ Assume that $\chi(\overline{N}_n\setminus N_{n-1})=-1$. Consider a smooth Jordan curve $\wh{\alpha}\in \Hcal_1(\overline{N}_n,\z)\setminus \Hcal_1(\overline{N}_{n-1},\z)$ contained in $N_n$ and intersecting $N_n\setminus N_{n-1}$ in a Jordan arc $\alpha$ with  endpoints $a,$ $b$ in $b \overline{N}_{n-1}$ and otherwise disjoint from $\overline{N}_{n-1}.$ Notice that, since $\overline{N}_{n-1}$ and $\overline{N}_n$ are Runge subsets of $\Ncal$ and $\chi(\overline{N}_n\setminus N_{n-1})=-1$, then  $\Hcal_1(\overline{N}_n,\z)=\Hcal_1(\overline{N}_{n-1}\cup \alpha,\z)$ and $\overline{N}_{n-1} \cup \alpha\subset\Ncal$ is Runge as well. 

Likewise, we choose a smooth Jordan arc $\gamma\subset  \Ncal\setminus U_{n-1}$ attached  transversally to $b\overline{U}_{n-1}$ at the points $\eta_{n-1}^{-1}(a)$ and $\eta_{n-1}^{-1}(b)$ and otherwise disjoint from $\overline{U}_{n-1}$. We take $\gamma$ such that there exists an isotopical homeomorphism $\tau\colon \overline{U}_{n-1} \cup \gamma \to \overline{N}_{n-1} \cup \alpha$ so that  $\tau|_{\overline{U}_{n-1}}=\eta_{n-1}$ and $\tau(\gamma)=\alpha.$

In $\c^2$, choose a smooth regular Jordan arc $\lambda\subset \Fr\Dcal$ attached transversally to $X_{n-1}(b\overline{U}_{n-1})$ at the points $X_{n-1}(\eta_{n-1}^{-1}(a))$ and $X_{n-1}(\eta_{n-1}^{-1}(b))$ and otherwise disjoint from $X_{n-1}(\overline{U}_{n-1})$.

From {\rm (6$_{n-1}$)} and the fact that $\lambda\subset \Fr\Dcal$, there exist a tangent net $\wh\Tcal_n$ of radius $<\epsilon_n$ for $\Dcal^{n-1}$ and a positive $\wh\delta_n<\epsilon_n$, such that
\begin{enumerate}[\rm (i$'$)]
\item[\rm (ii$'$)] $X_{n-1}(b\overline{U}_{n-1})\cup\lambda \Subset \wh\Tcal_n$ and
\item[\rm (iv$'$)] $\ell(\gamma)> {\bf d}(\Dcal^{n-1},\Fr\Dcal^n)-\epsilon_n$ for any Jordan arc $\gamma\subset \wh\Tcal_n$ connecting $\Fr\Dcal^{n-1}_{\wh\delta_n}$ and $\Fr\Dcal^n$.
\end{enumerate}

Extend $X_{n-1}$, with the same name, to a smooth function $\overline{U}_{n-1}\cup \gamma \to\c^2$ mapping $\gamma$ diffeomorphically to $\lambda$. In this setting, Mergelyan's theorem furnishes a bordered domain $V_{n-1}\subset\Ncal$ with $U_{n-1}\cup\gamma \Subset V_{n-1}\Subset U_n$, $\chi(\overline{U}_n\setminus V_{n-1})=0$, and a holomorphic immersion $\wh X_{n-1}\colon \overline{V}_{n-1}\to\c^2$, as close as desired to $X_{n-1}$ in the $\Ccal^0$ topology on $\overline{U}_{n-1}\cup \gamma$ and in the $\Ccal^1$ topology on $\overline{U}_{n-1}$, such that $\wh X_{n-1}(b\overline{V}_{n-1})\subset \wh\Tcal_n\cap \Dcal^{n-1}_{\wh\delta_n}$. We finish by using Lemma \ref{lem:main} as in the previous case for small enough $\wh\delta_n$.

This concludes the construction of the sequence $\{\Upsilon_n\}_{n\in\n}$.

Set $U:=\cup_{n\in\n} U_n$. For condition Theorem \ref{th:extrinsic}-{\rm (A)}, use {\rm (2$_n$)}, $n\in\n$, and the fact that $\{N_n\}_{n\in\n}$ is an exhaustion of $\Ncal$; take into account that $\Mcal=U_0$.

From {\rm (4$_n$)} and {\rm (3$_n$)}, $n\in\n$, the sequence $\{X_n\}_{n\in\n}$ converges uniformly on compact subsets of $U$ to a holomorphic function $Y\colon U\to\c^2$ satisfying item {\rm (B)}.

Let us check that $Y$ meets all the requirements in the theorem.

\noindent$\bullet$ $Y$ is an immersion. Indeed, for any $k\in\n$, properties {\rm (3$_n$)} and {\rm (4$_n$)}, $n>k$, give that 
\begin{equation}\label{eq:cauchy}
\|Y-X_k\|_{1,\overline{U}_k;\vartheta_\Ncal}\leq \sum_{n>k} \|X_n-X_{n-1}\|_{1,\overline{U}_k;\vartheta_\Ncal}<\sum_{n>k}\epsilon_n<2\epsilon_{k+1}<\epsilon_k;
\end{equation}
hence the latter assertion in {\rm (3$_n$)} gives that $Y|_{\overline{U}_k}$ is an immersion for all $k\in\n$, and so is $Y$.
%Indeed, fix $P\in U$, take $k\in\n$ with $P\in\overline{U}_k$, and notice that 
%\begin{eqnarray*}
%\|dY/\vartheta_\Ncal\|(P) & \geq & \|dX_k/\vartheta_\Ncal\|(P)-\sum_{n>k}\|(dX_n-dX_{n-1})/\vartheta_\Ncal\|(P)\\
 %& \stackrel{\text{\rm (4$_n$), (3$_n$)}}\geq & \|dX_k/\vartheta_\Ncal\|(P)-\sum_{n>k}\rho_n/2^{n+1}\\
 %& \stackrel{\text{\eqref{eq:rhoN}}}\geq & \|dX_k/\vartheta_\Ncal\|(P)(1-\sum_{n>k}1/2^{n+1})\;>\;0,\\
%\end{eqnarray*}
%where in the latter inequality we have used that $X_k\colon\overline{U}_k\to\c^2$ is an immersion.

\noindent$\bullet$ $Y(U)\subset\Bcal$ and $Y\colon U \to \Bcal$ is proper. We proceed like in the proof of Theorem \ref{th:embe}. Up to taking limit as $n \to \infty$, the assertion $Y(U)\subset\Bcal$ follows from  {\rm (6$_n$)} and the Convex Hull Property.
Likewise, properties {\rm (5$_n$)}, $n\in\n$, and the fact that  $\{ \overline{\Dcal^{n-1}_{-\epsilon_n}} \}_{n \in \n}$ is an exhaustion by compact sets of $\Bcal$ imply that 
\begin{equation}\label{eq:CD}
Y(U \setminus \overline{U}_{k-1}) \subset \Bcal \setminus \Dcal^{k-1}_{-\epsilon_k}\quad \text{for all $k\in\n$.}
\end{equation}
This inclusion for $k=1$ proves {\rm (D)}. The properness of  $Y\colon U \to \Bcal$ follows from the fact that $\{ \overline{\Dcal^{n-1}_{-\epsilon_n}} \}_{n \in \n}$ is an exhaustion of $\Bcal$ and \eqref{eq:CD}. This concludes {\rm (C)}.

\noindent$\bullet$ $Y$ is image complete. Indeed, let $\alpha$ be a locally rectifiable divergent arc in $Y(U)$, and let us check that $\ell(\alpha)=\infty$. Since $Y\colon U\to\Bcal$ is proper, then $\alpha$ is a divergent arc in $\Bcal$ as well. Let $n_0\in\n$ large enough so that the initial point of $\alpha$ lies in $\Dcal^{n_0}$. For every $a\in \n$, $a>n_0$,  let $\alpha_a$ denote a compact sub-arc of $\alpha$ in $\overline{\Dcal^a}\setminus \Dcal^{a-1}$ connecting   $\Fr \Dcal^{a-1}$ and $\Fr \Dcal^a$. Since $\{\Dcal^n\}_{n\in\n}$ is {\bf d}-proper  in $\Bcal$ (see Def.\ \ref{def:proper}) and $\sum_{n\in\n}\epsilon_n$ converges, then it suffices to show that $\ell(\alpha_a)\geq {\bf d}(\Dcal^{a-1},\Fr\Dcal^a)-\epsilon_a$ for all $a>n_0$. 

Indeed, fix $a>n_0$. Let $n_1\in\n$, $n_1\geq a$, large enough so that $\alpha_a\subset Y({U}_{n_1})$; recall that $Y\colon U\to\Bcal$ is proper. Let $\beta_a=\cup_{j=1}^k \beta_{a,j} \subset {U}_{n_1}$ be a finite union of compact arcs with $Y(\beta_a)=\alpha_a$.
Without loss of generality,  we can suppose that the arcs $\{\alpha_{a,j}:=Y(\beta_{a,j})\colon j=1,\ldots,k\}$ are laid end to end and the endpoints of  $\alpha_{a,j}$,  $j=2,\ldots,k-1$, are double points of $Y(\overline{U}_{n_1})$.

Since  the double points of $Y$ are isolated and stable under deformations and $\{\|Y-X_n\|_{1,\overline {U}_{n_1};\vartheta_\Ncal}\}_{n\geq n_1}\to 0$  (see \eqref{eq:cauchy}),  for any sufficiently large $n\geq n_1$  we can find compact arcs 
$ {\beta}^n_{a,j}$, $j=1,\ldots,k$, in ${U}_{n_1}$ such that
\begin{itemize}
\item  $ {\alpha}^n_a:=X_n( {\beta}^n_{a})$ is a Jordan arc in $\overline{\Dcal^a}\setminus \Dcal^{a-1}$ connecting   $\Fr \Dcal^{a-1}$ and $\Fr \Dcal^a$, where $ {\beta}^n_{a}=\cup_{j=1}^k  {\beta}^n_{a,j}$, and
\item $\{\ell( {\alpha}^n_a)\}_{n>n_1} \to \ell(\alpha_a)$.
\end{itemize}
To see this, just observe that the double points of $X_n|_{U_{n_1}}$ converge to the ones of $Y|_{{U}_{n_1}}$ as $n\to \infty$, and choose $ {\beta}^n_{a,j}$ as a sufficiently slight deformation of $\beta_{a,j}$ in $U_{n_1}$ so that  $\{ {\alpha}^n_{a,j}=X_n( {\beta}^n_{a,j}) \colon j=1,\ldots,k\}$ are laid end to end, the endpoints of  ${\alpha}^n_{a,j}$,  $j=2,\ldots,k-1$, are double points of $X_n(U)$,  and  $\{\ell( X_n({\beta}^n_{a,j}))- \ell( X_n({\beta}_{a,j}))\}_{n>n_1} \to 0$.

By property {\rm (7$_n$)}, $\ell(\alpha^n_a)>{\bf d}(\Dcal^{a-1},\Fr\Dcal^a)-\epsilon_a$ for  any large enough $n\geq n_1$. Taking limits as $n\to\infty$,  $\ell(\alpha_a)=\ell(Y(\beta_a))\geq {\bf d}(\Dcal^{a-1},\Fr\Dcal^a)-\epsilon_a$ as claimed.

This shows item {\rm (E)} and concludes the proof of the theorem.
\end{proof}

\subsection*{Added in Proof} After this paper was written, Globevnik \cite{Glo3,Glo4}, with a different method, proved that every pseudoconvex domain in $\c^n$, for any $n\ge 2$, contains a complete closed complex hypersurface; in particular, this answers in the optimal way the question just below Corollary \ref{cor:iii} in which concerns assertion {\rm (i)}. More recently, Alarc\'on, Globevnik, and L\'opez \cite{AGL}, also with a new different method, constructed complete closed complex hypersurfaces in the unit ball of $\c^n$, for any $n\ge 2$, with certain control on the topology; in particular, they affirmatively answered Question \ref{qu:ex} by giving examples with any finite topology.

\subsection*{Acknowledgments}
A.\ Alarc\'on is supported by the Ram\'on y Cajal program of the Spanish Ministry of Economy and Competitiveness, he is also partially supported by MCYT-FEDER grants MTM2007-61775 and MTM2011-22547, MINECO/FEDER grant no. MTM2014-52368-P, Junta de Andaluc\'{i}a Grant P09-FQM-5088, and the grant PYR-2012-3 CEI BioTIC GENIL (CEB09-0010) of the MICINN CEI Program, Spain.

F.\ J.\ L\'{o}pez is partially supported by MCYT-FEDER research projects MTM2007-61775 and MTM2011-22547, MINECO/FEDER grant no. MTM2014-52368-P, and Junta de Andaluc\'{\i}a Grant P09-FQM-5088, Spain.

The authors wish to thank Franc Forstneri\v c for helpful discussions about the paper.

%%%%%%%%%%%%%%%%%
%%%%%%%%%%%%%%%%%
%%%%%%%%%%%%%%%%%

\end{document}